\documentclass[12pt,a4paper,reqno]{amsart}
\usepackage[T1]{fontenc}
\usepackage{cleveref}
\usepackage{amsmath,amssymb,ifthen}
\usepackage{mathrsfs}
\usepackage{verbatim}
\usepackage[a4paper, portrait, margin=28mm]{geometry}
\usepackage{url}
\usepackage{graphicx,psfrag,subfigure}
\usepackage{color}

\def\R{{\mathbb R}}
\def\N{{\mathbb N}}

\def\bb{{\boldsymbol{b}}}

\def\bu{\boldsymbol{u}}

\def\AA{{\mathcal A}}
\def\BB{{\mathcal B}}
\def\CC{{\mathcal C}}

\def\HH{{\mathcal H}}

\def\NN{{\mathcal N}}

\def\OO{{\mathcal O}}
\def\PP{{\mathcal P}}
\def\RR{{\mathcal R}}

\def\TT{{\mathcal T}}

\def\AAb{\boldsymbol{\AA}}
\def\BBb{\boldsymbol{\BB}}
\def\PPb{\boldsymbol{\PP}}
\def\CCb{\boldsymbol{\CC}}
\def\Ub{\boldsymbol{U}}
\def\ub{\boldsymbol{u}}
\def\fb{\boldsymbol{f}}
\def\Ab{\boldsymbol{A}}
\def\ab{\boldsymbol{a}}
\def\bb{\boldsymbol{b}}
\def\Rb{\boldsymbol{R}}
\def\Tb{\boldsymbol{T}}
\def\Sb{\boldsymbol{S}}
\def\Fb{\boldsymbol{F}}

\def\alphab{\boldsymbol{\alpha}}
\def\betab{\boldsymbol{\beta}}

\def\E{{\mathbb E}}

\def\norm#1#2{\|#1\|_{#2}}

\def\set#1#2{\big\{#1\,:\,#2\big\}}

\def\id{\mathfrak{1}}

\newcommand{\enorm}[2][]{#1|\hspace*{-.5mm}#1|\hspace*{-.5mm}#1|#2#1|\hspace*{-.5mm}#1|\hspace*{-.5mm}#1|}

\def\normL2#1#2{\|#1\|_{L^2(#2)}}

\newtheorem{theorem}{Theorem}

\newtheorem{lemma}[theorem]{Lemma}
\newtheorem{corollary}[theorem]{Corollary}

\newtheorem{remark}[theorem]{Remark}

\def\id{{\rm id}}

\def\level{\operatorname{lev}}

\numberwithin{equation}{section}

\begin{document}
\title{
Sparse Compression of Expected Solution Operators
}
\author{Michael Feischl}
\address{TU Wien}
\email{michael.feischl@tuwien.ac.at}

\author{Daniel Peterseim}
\address{Universit\"at Augsburg}
\email{daniel.peterseim@math.uni-augsburg.de}
\thanks{The authors would like to acknowledge the kind hospitality of the Erwin Schr\"odinger International Institute for Mathematics and Physics (ESI) where large parts of this research were developed under the frame of the Thematic Programme \emph{Numerical Analysis of Complex PDE Models in the Sciences}. D.~Peterseim acknowledges support by Deutsche Forschungsgemeinschaft in the Priority Program 1748 {\it Reliable simulation techniques in solid mechanics} (PE2143/2-2).
M.~Feischl acknowledges support by Deutsche
Forschungsgemeinschaft (DFG) through CRC 1173.}

\date{\today}

\begin{abstract}
We show that the expected solution operator of prototypical linear elliptic partial differential equations with random coefficients is well approximated by a computable sparse matrix. This result is based on a random localized orthogonal multiresolution decomposition of the solution space that allows both the sparse approximate inversion of the random operator represented in this basis as well as its stochastic averaging. The approximate expected solution operator can be interpreted in terms of classical Haar wavelets. When combined with a suitable sampling approach for the expectation, this construction leads to an efficient method for computing a sparse representation of the expected solution operator.  
\end{abstract}
\maketitle
\thispagestyle{empty}
\section{Introduction}
For a random (or parameterized) family of prototypical linear elliptic partial differential operators $\AAb(\omega)=-\operatorname{div}(\Ab(\omega)\nabla\bullet)$ and a given deterministic right-hand side $f$, we consider the family of solutions
\begin{align*}
\ub(\omega) := \AAb(\omega)^{-1}f
\end{align*}
with events $\omega\in \Omega$ in some probability space $\Omega$.
We define the harmonically averaged operator  
\begin{align*}
\AA:= \Big(\E[\AAb(\omega)^{-1}]\Big)^{-1}.
\end{align*}
The idea behind this definition is that $\E(\ub)$ satisfies
\begin{align*}
 \E[\ub]=\AA^{-1} f.
\end{align*}
In this sense, $\AA$ may be understood as a stochastically homogenized operator and $\AA^{-1}$ is the 
effective solution operator. Note that this definition does not rely on probabilistic structures of the random 
diffusion coefficient $\Ab$ such as stationarity, ergodicity or any characteristic length of correlation.
However, we shall emphasize that $\AA$ does not coincide with the partial differential operator that would result 
from the standard theory of stochastic 
homogenization (under stationarity and ergodicity) \cite{Kozlov1979,PapanicolaouVaradhan1981,Yurinskii1986}; see e.g. 
\cite{BourgeatPiatnitski2004}, \cite{GloriaOtto2015,DuerinckxGloriaOtto2016,GloriaOtto2017}, 
\cite{ArmstrongKuusiMourrat2017} for quantitative results.
Recent works on discrete random problems on $\mathbb{Z}^d$ with
iid edge conductivies indicate that $\AA$ is rather a non-local perturbation of the Laplacian by a convolution type operator \cite{Bourgain2018,2018arXiv180410260K,Duer19}. The goal of the present work is to show that, even in the more general PDE setup of this paper without any assumptions on the distribution of the random coefficient, the expected solution operator $\AA^{-1}$ can be represented accurately by a sparse matrices $R^\delta$ in the sense that 
$$\|\AA^{-1}-R^\delta\|_{L^2(D)\rightarrow L^2(D)}\leq \delta$$ 
for any $\delta>0$ while the number of non-zero entries of $R^\delta$ scales like $\delta^{-d}$ up to logarithmic-in-$\delta$ terms (see Theorem~\ref{t:main}). 
\medskip

The sparse matrix representation of $\AA^{-1}$ is based on multiresolution decompositions of the energy space in the spirit of numerical homogenization by localized orthogonal decomposition (LOD)  \cite{MalP14,HenningPeterseim2013,Peterseim2016,GallistlPeterseim2017,KorPY18,GallistlPeterseim2017random} and, in particular, its multi-scale generalization that is popularized under the name gamblets \cite{gamblets}. In this paper, a one-to-one correspondence of a gamblet decompostion and classical Haar wavelets is established via $L^2$-orthogonal projections and conversely by corrections involving the solution operator (see Section~\ref{s:wavelets}). The resulting problem-dependent multiresolution decompositions block-diagonalize the random operator $\AAb$ for any event in the probability space (see Section~\ref{sec:hierarchical}). The block-diagonal representations (with sparse blocks) are well conditioned and, hence, easily inverted to high accuracy using a few steps of standard linear iterative solvers. The sparsity of the inverted blocks is preserved to the degree that it deteriorates only logarithmically with higher accuracy. 

While the sparsity pattern of the inverted block-diagonal operator is independent of the stochastic parameter and, hence, not affected when taking the expectation (or any sample mean) the resulting object cannot be interpreted in a known basis. This issue is circumvented by reinterpreting the approximate inverse stiffness matrices in terms of the deterministic Haar basis before stochastic averaging (see Section~\ref{sec:transform}). This leads to an accurate representation of $\AA^{-1}$ in terms of piecewise constant functions. Sparsity is not directly preserved by this transformation but can be retained by some appropriate hyperbolic cross truncation which is justified by scaling properties of the multiresolution decomposition (see Section~\ref{s:inverse}). 
\medskip

Apart from the mathematical question of sparse approximability of the expected operator, the above construction leads to a computationally efficient method for approximating $\AA^{-1}$ when combined with any sampling approach for the approximation of the expectation (see Section~\ref{s:sos}). 
This new sparse compression algorithm for the direct discretization of $\AA^{-1}$ may be beneficial if we want to compute $\E[\ub]$ for multiple right-hand sides $f$. This, for example, is the case if we have an independent probability space $\xi\in\Xi$ influencing $f=\fb({\xi})$
as well as the corresponding solution $\Ub(\omega,\xi) :=\AAb(\omega)^{-1}\fb(\xi)$.
Then, we might be interested in the average behavior $\E_{\Omega\times \Xi}[\Ub]$
which is the solution of
\begin{align}\label{eq:darcy}
\E_{\Omega\times \Xi}[\Ub]=\E_{\Xi}[\AA^{-1}\fb]=\AA^{-1}\E_{\Xi}[\fb].
\end{align}
While this can be computed efficiently with sparse approximations of the random parameter (see, e.g.,~\cite{coh2,coh1}) or 
multi-level algorithms (see, e.g.,~\cite{homl,giles}) under regularity assumption on the random parameter, the present
approach does not assume any smoothness apart from integrability. As a practical example for the problem might serve the Darcy flow as a model of ground water flow. Here, $\AAb$ is a random diffusion process modeling the unknown diffusion coefficient of the ground material.
The right-hand side $\fb$ would be the random (unknown) injection of pollutants into the ground water. Ultimately, the user would be interested in the average 
distribution of pollutants in the ground. Obviously, computing the right-hand side of~\eqref{eq:darcy} requires the user
to sample $\Omega$ and $\Xi$ successively, whereas computing the left-hand side of~\eqref{eq:darcy} forces the
user to sample the much larger product space $\Omega\times \Xi$. While for plain Monte Carlo sampling only 
the possibly increased variance of the product random variable effects the convergence, 
higher-order sampling methods such as sparse-grids and quasi-Monte Carlo will directly 
(and in case of lack of regularity on the random parameter quite drastically, see, e.g., exponential 
dependence on dimension in~\cite{dick}) benefit from the reduction of dimension of the probability space.
Therefore, an accurate discretization of $\AA$ can help saving significant computational cost. 

We consider some prototypical linear second order elliptic partial differential equation with random diffusion coefficient. Let $(\Omega,\mathcal{F},\mathbb{P})$ be a probability space with set of events $\Omega$, $\sigma$-algebra $\mathcal{F}\subseteq 2^\Omega$ and probability measure $\mathbb{P}$. 
The expectation operator is denoted by $\mathbb{E}$.
Let $D\subseteq\mathbb R^d$ for $d\in\{1,2,3\}$ be
a bounded Lipschitz polytope with diameter of order $1$. The set of admissible coefficients reads
\begin{equation*}\label{e:classM}
\mathcal{M}(D,\gamma_{\operatorname{min}},\gamma_{\operatorname{max}}) 
= \left\{\begin{aligned}A\in 
L^\infty(D;\mathbb{R}^{d\times d}_{\mathrm{sym}})&\text{ s.t. }\,\gamma_{\operatorname{min}}|\xi|^2 \leq (A(x)\xi)\cdot\xi \leq \gamma_{\operatorname{max}} |\xi|^2\\&\text{for a.e. }x\in D\text{ and all }\xi\in\mathbb{R}^d
\end{aligned}
\right\}
\end{equation*}
for given uniform spectral bounds $0<\gamma_{\operatorname{min}}\leq\gamma_{\operatorname{max}}<\infty$. Here, $\mathbb{R}^{d\times d}_{\mathrm{sym}}$ denotes
the set of symmetric $d\times d$ matrices. 
Let $\Ab$ be a Bochner-measurable $\mathcal{M}(D,\gamma_{\operatorname{min}},\gamma_{\operatorname{max}})$-valued random field with $\gamma_{\operatorname{max}}>\gamma_{\operatorname{min}}>0$. Note that we do not make any structural assumptions regarding the distribution of $\Ab$. Moreover, realizations in  $\mathcal{M}(D,\gamma_{\operatorname{min}},\gamma_{\operatorname{max}})$ are fairly free to vary within the bounds $\gamma_{\operatorname{min}}$ and $\gamma_{\operatorname{max}}$ without any conditions on frequencies of variation or smoothness. 

Denote the energy space by $V:=H^1_0(D)$ and let $f\in  V^*=H^{-1}(D)$ be deterministic. The prototypical second order elliptic variational problem seeks a $V$-valued random field $\ub$ such that, for almost all $\omega\in\Omega$,
\begin{equation}\label{e:modelweak}
\ab_\omega(\ub(\omega),v):=\int_D(\Ab(\omega)(x)\nabla\ub(\omega)(x))\cdot\nabla v(x)\,dx = f(v)\quad\text{for all }v\in V.
\end{equation}
The bilinear from $\ab_\omega$ depends continuously on the coefficient $\Ab(\omega)\in \mathcal{M}(D,\gamma_{\operatorname{min}},\gamma_{\operatorname{max}})$ and, particularly, is measurable as a function of $\omega$. Hence,
the reformulation of this problem in the Hilbert space $L^2(\Omega;V)$ of $V$-valued random fields with finite second moments shows well-posedness in the sense that there exists a unique solution 
$\ub\in L^2(\Omega;V)$ with
\begin{equation*}
\|\nabla\ub\|_{L^2(\Omega;V)}
:=\left(\int_\Omega\int_D |\nabla(\ub(\omega))(x)|^2\,dx \,d\mathbb{P}(\omega)\right)^{1/2}
\leq \gamma_{\operatorname{min}}^{-1}\|f\|_{V^*}. 
\end{equation*}
To connect the model problem to the operator setting of the introduction, we shall  define the random operator $\AAb\colon \Omega\to\mathcal{L}(V,V^*)$ by
\[
\langle \AAb(\omega) u, v \rangle_{V^*,V} := \ab_\omega(u, v)
\]  
for functions $u,v\in V$ and $\omega\in\Omega$. Then the model problem \eqref{e:modelweak} can be rephrased as
\[
\AAb(\omega) \ub(\omega) = f\quad\text{for almost all }\omega\in\Omega. 
\]
For convenience, we define the sample-dependent energy norm $\enorm{\cdot}_{\omega}^2:= \ab_\omega(\cdot,\cdot)$.
\section{Coefficient-adapted hierarchical bases}\label{s:wavelets}
Let $\TT_\ell$, $\ell=0,\ldots,L$ denote a sequence of uniform refinements with mesh-size $h_\ell$
of some initial mesh $\TT_0$ of $D$ and let $\NN(\TT_\ell)$ denote the nodes of the meshes. We allow fairly general meshes in the sense that we only require a reference element $T_{\rm ref}$ together with a family of uniformly bi-Lipschitz maps $\Psi_T\colon T_{\rm ref}\to T$ for all elements $T\in\TT_\ell$, $\ell=0,\ldots,L$.
Straightforward examples are simplicial meshes generated from an initial triangulation by red refinement (or newest vertex bisection)
or quadrilateral meshes generated by subdividing the elements into $2^d$ new elements. Particularly, hanging nodes do not pose problems as long as the other properties are observed.
 
The number of levels (or scales) $L$ will typically be chosen proportional to the modulus of some logarithm of the desired accuracy $1\gtrsim\delta>0$. We assume $h_{\ell+1}\leq  h_\ell/2$. Note that any other fixed factor of mesh width reduction strictly smaller than one would
do the job. Define the set of descendants of an element $T\in\TT_\ell$ by ${\rm ref}(T):=\set{T^\prime\in\TT_{\ell+1}}{T^\prime\subseteq T}$.
For each $T\in\bigcup_{\ell=0}^{L-1}\TT_\ell$, we pick piecewise constant functions $\phi_{T,1},\phi_{T,2},\ldots,\phi_{T,\#{\rm ref}(T)}\in P^0({\rm ref}(T))$ such that they are pairwise $L^2(T)$-orthogonal and $\int_T\phi_{T,j}\,dx=0$ for all $j=1,\ldots, \#{\rm ref}(T)$. With the indicator functions $\chi_{(\cdot)}$, we then define $\HH_0:=\set{\chi_T}{T\in\TT_0}$ and for $\ell\geq 1$
\begin{align}\label{e:Haarell}
	\HH_\ell:=\bigcup_{T\in\TT_{\ell-1}}\set{\phi_{T,j}}{j=1,\ldots, \#{\rm ref}(T)}. 
\end{align}
We define a Haar basis via \begin{equation*}
\HH:=\bigcup_{\ell=0}^L\HH_\ell.
\end{equation*}

\begin{lemma}\label{lem:haar}
	The basis $\HH$ is $L^2$-orthogonal and local in the sense that $\phi\in \HH_\ell$ satisfies ${\rm supp}(\phi)=T$ for some $T\in\TT_{\ell-1}$ or $T\in\TT_0$ for $\ell=0$.
\end{lemma}
\begin{proof}
Let $\phi_\ell\in \HH_\ell$ and $\phi_k\in \HH_k$. If $k=\ell$ then the interiors of the supports of any $\phi_k\neq\phi_\ell\in\HH_k$ are disjoint which implies $L^2(D)$ orthogonality. If $k<\ell$, we have that $\phi_k$ is constant on ${\rm supp}(\phi_\ell)$. Since $\int_D \phi_\ell\,dx=0$ by definition, this concludes the proof of $L^2$-orthogonality. Locality follows readily from the construction.
\end{proof}
\begin{remark}
For uniform Cartesian meshes, $\HH$ is the Haar basis. The choice of the $2^d-1$ generating functions follows the standard procedure for Haar wavelets (see e.g. \cite{MR1436437}). The construction is applicable to general meshes that are not based on tensor-product structures.
\end{remark}
\begin{figure}
	\includegraphics[height=0.12\textheight]{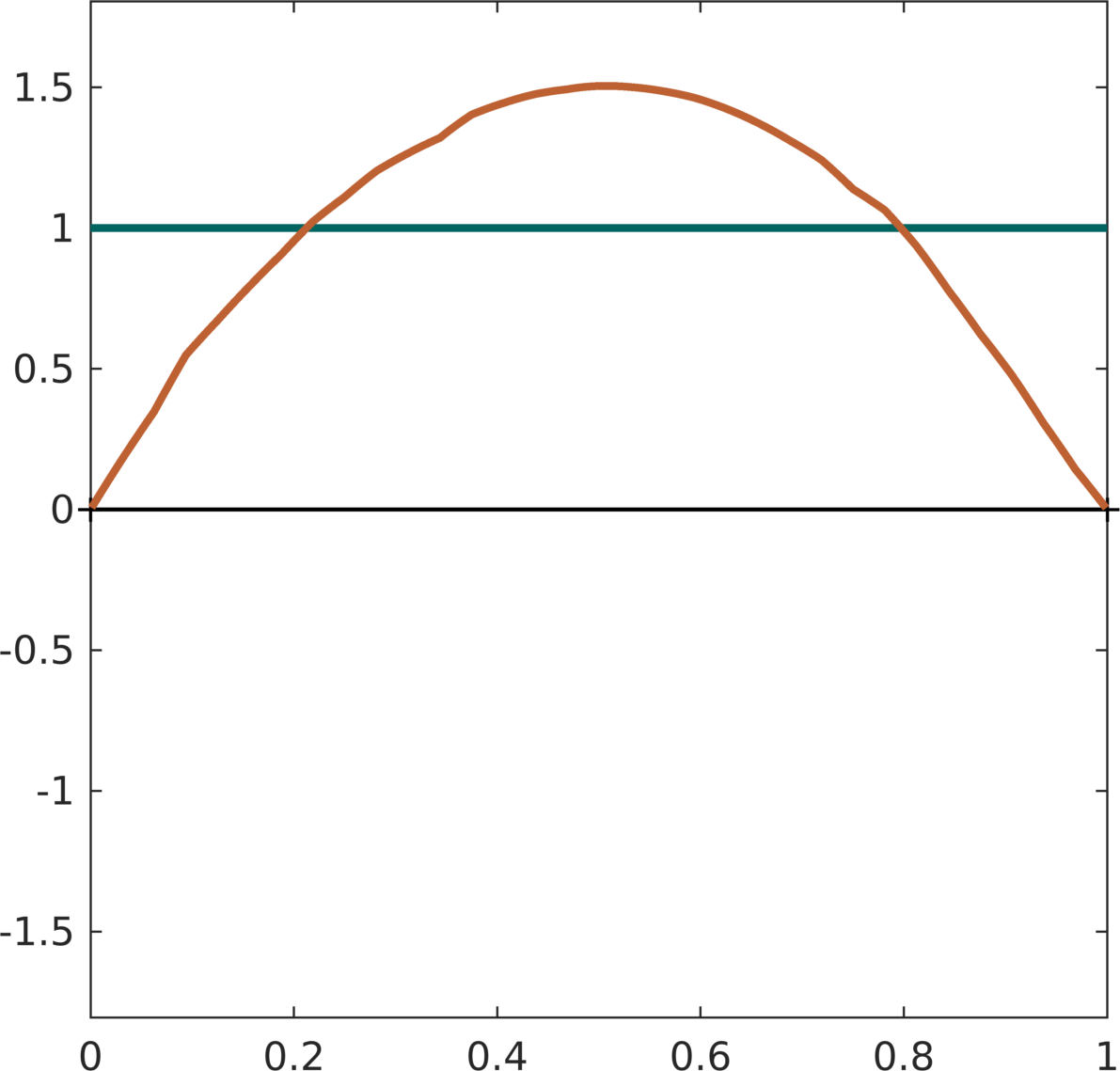}
	\includegraphics[height=0.12\textheight]{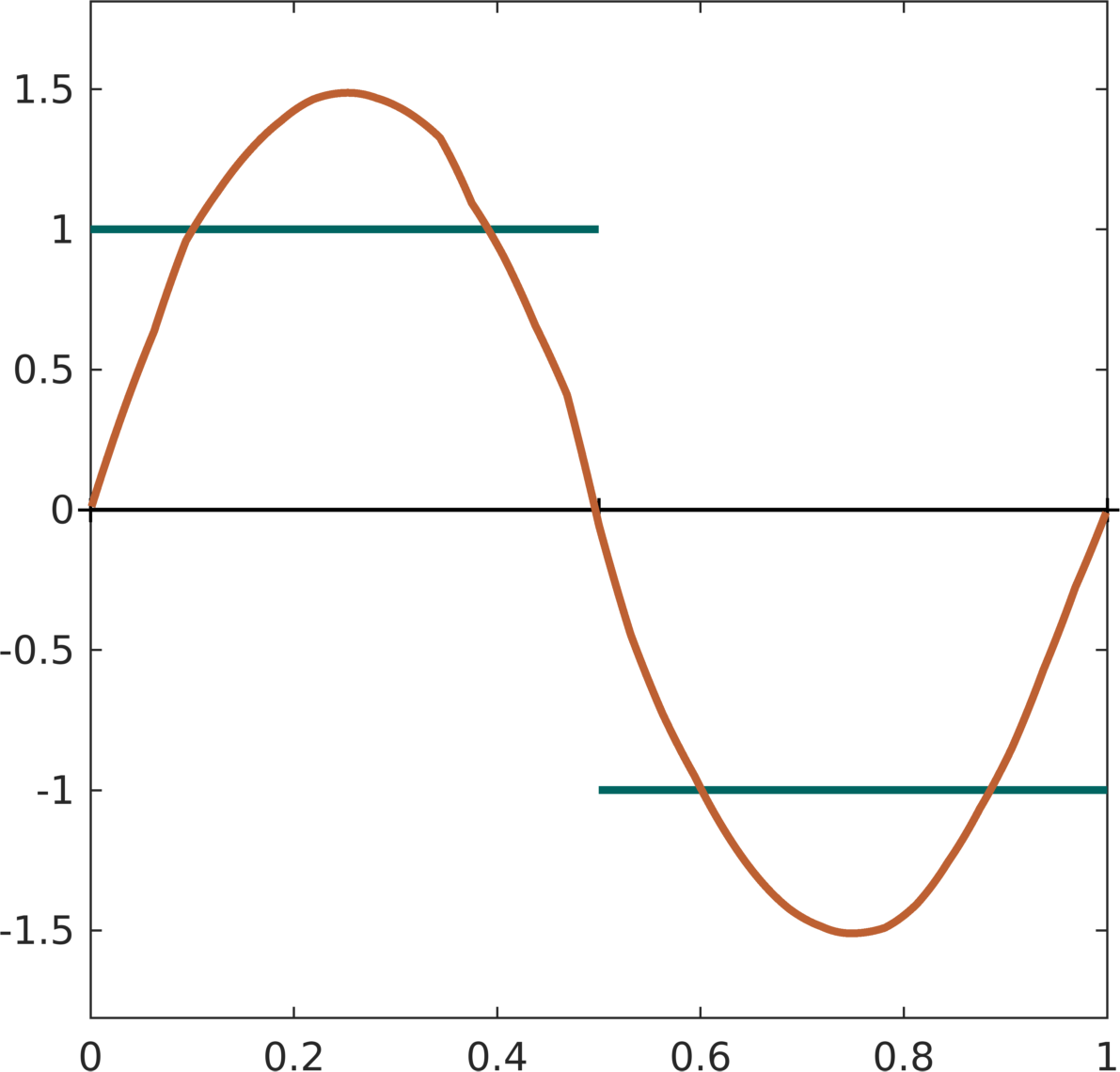}
	\includegraphics[height=0.12\textheight]{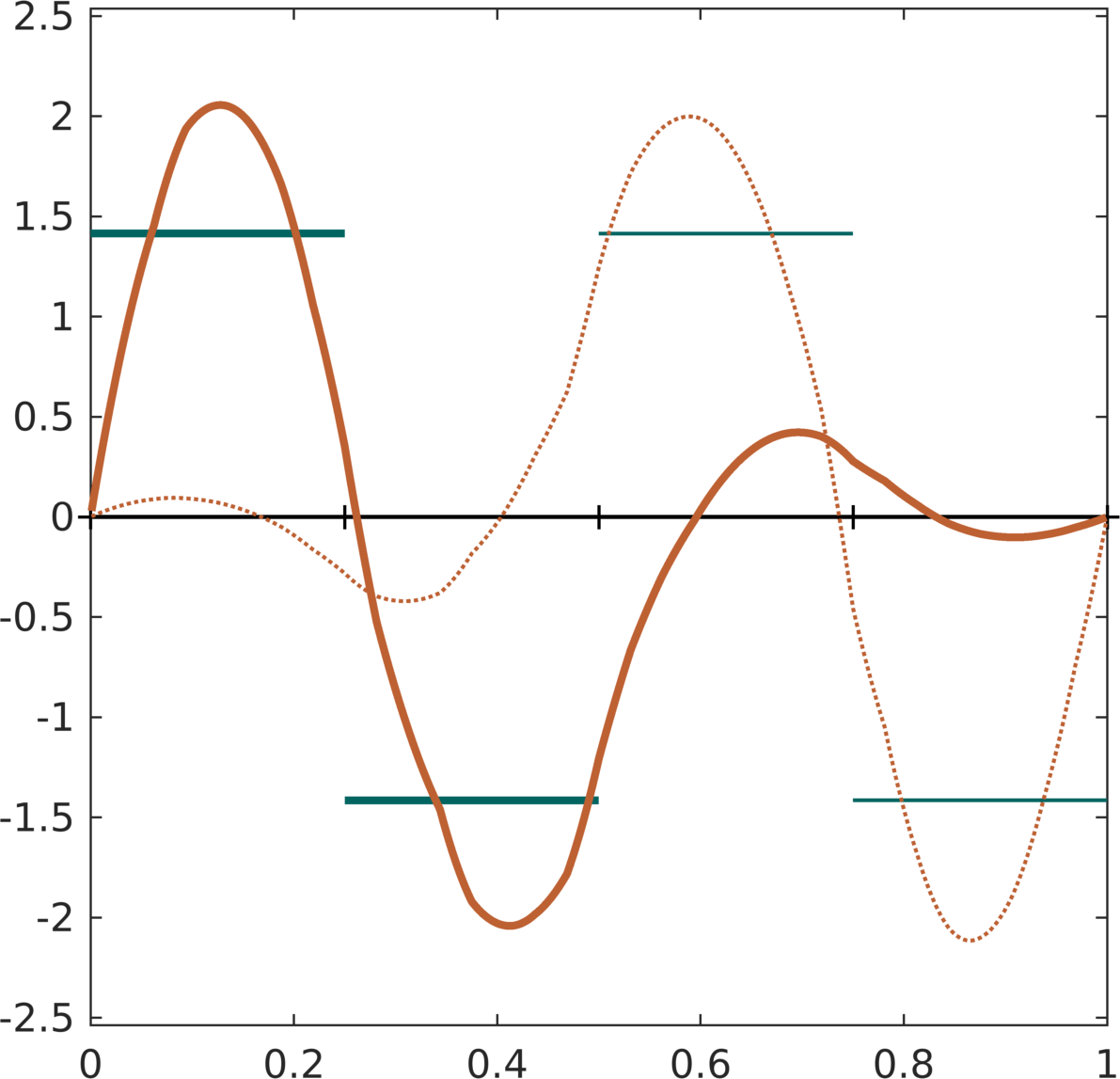}
	\includegraphics[height=0.12\textheight]{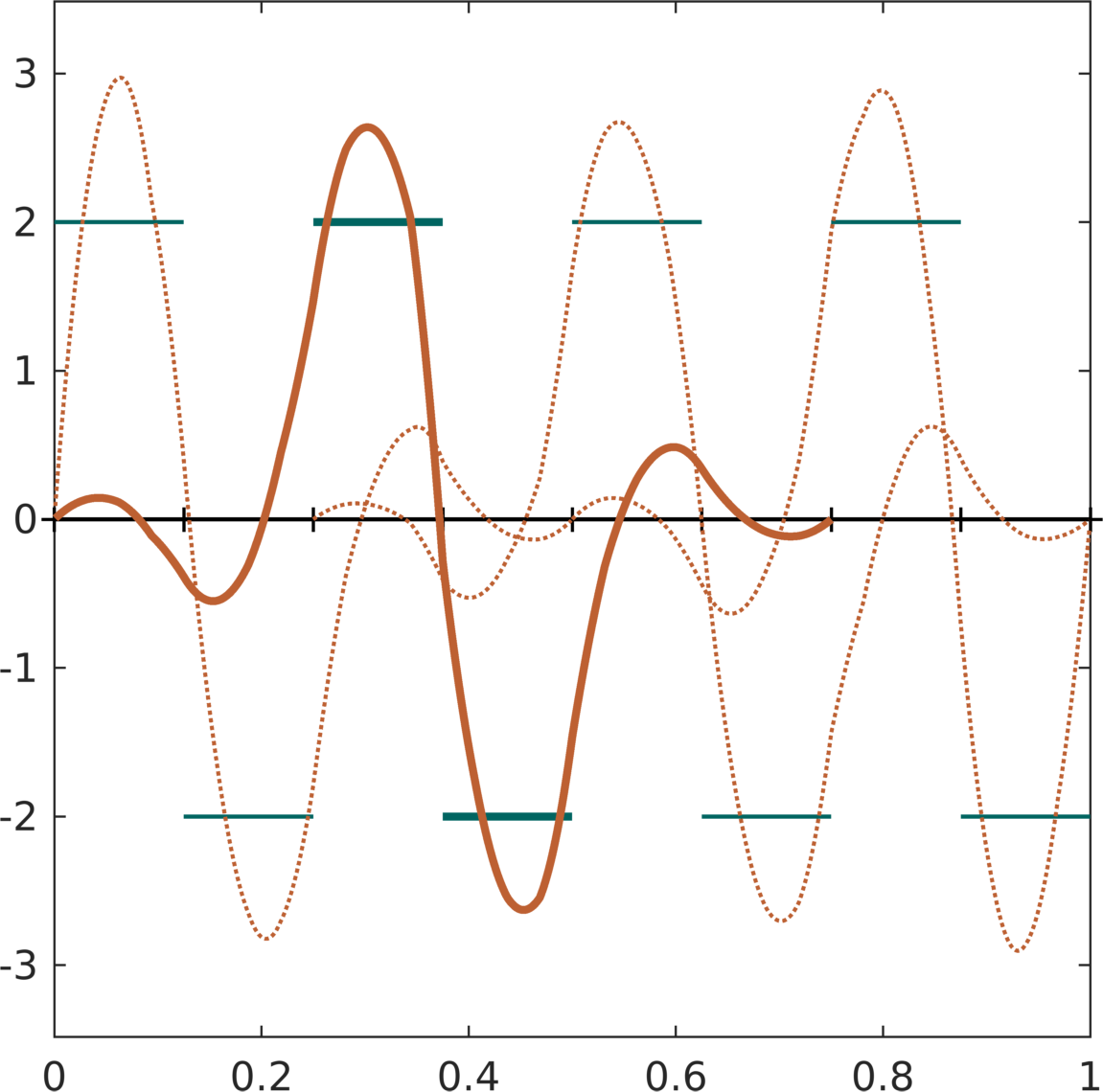}	\includegraphics[height=0.12\textheight]{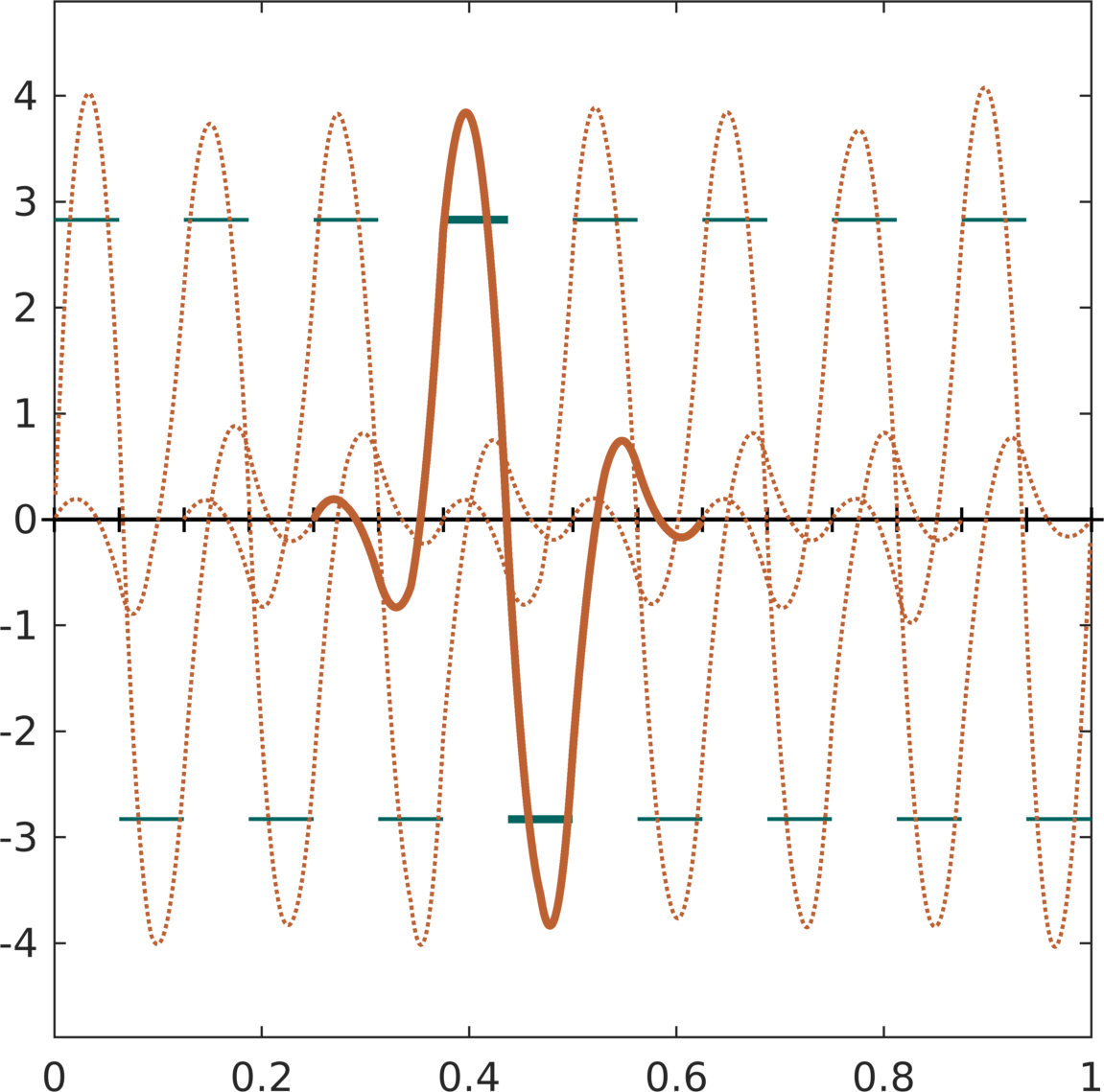}\\[10pt]
	\includegraphics[height=0.118\textheight]{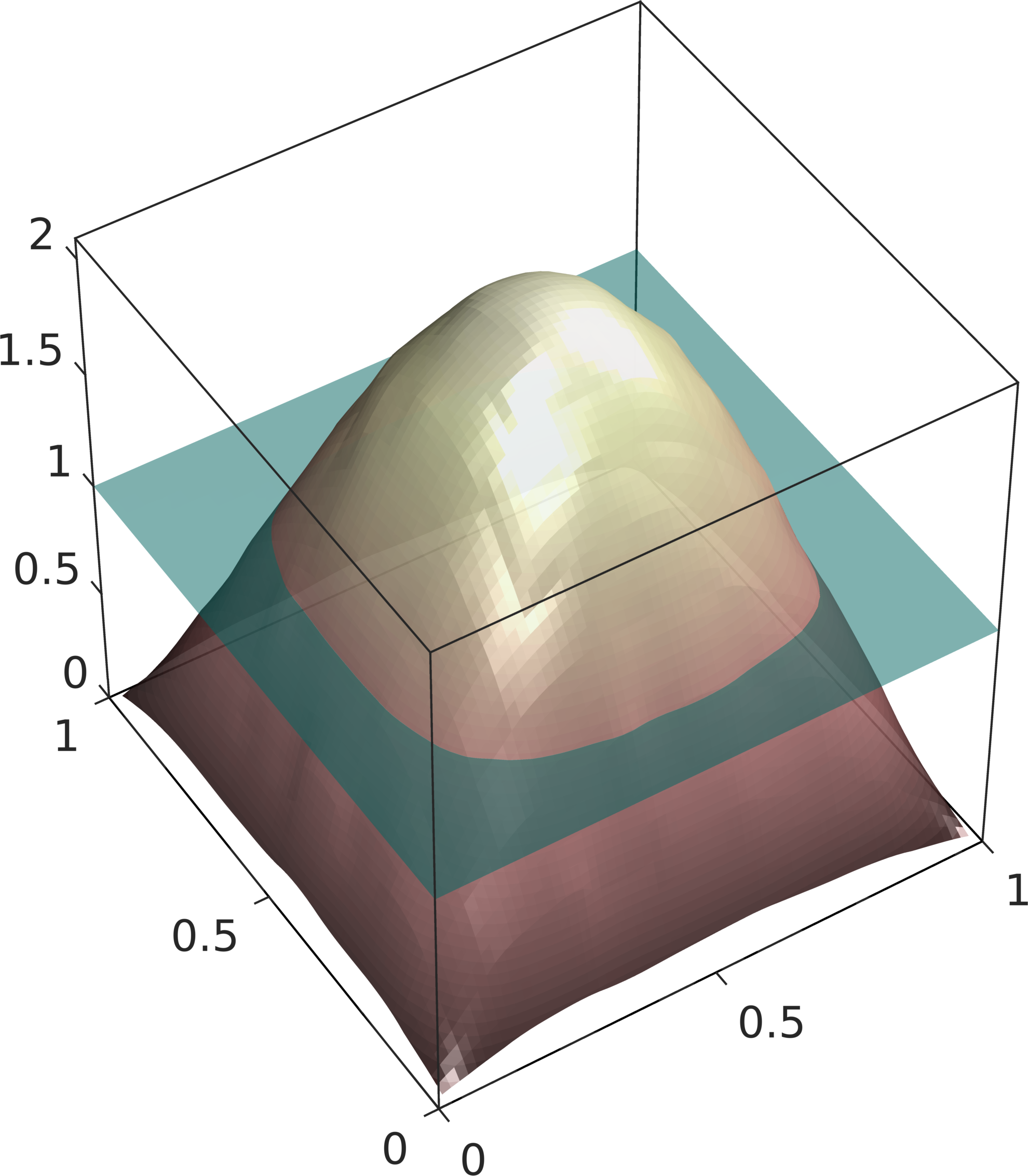}
	\includegraphics[height=0.118\textheight]{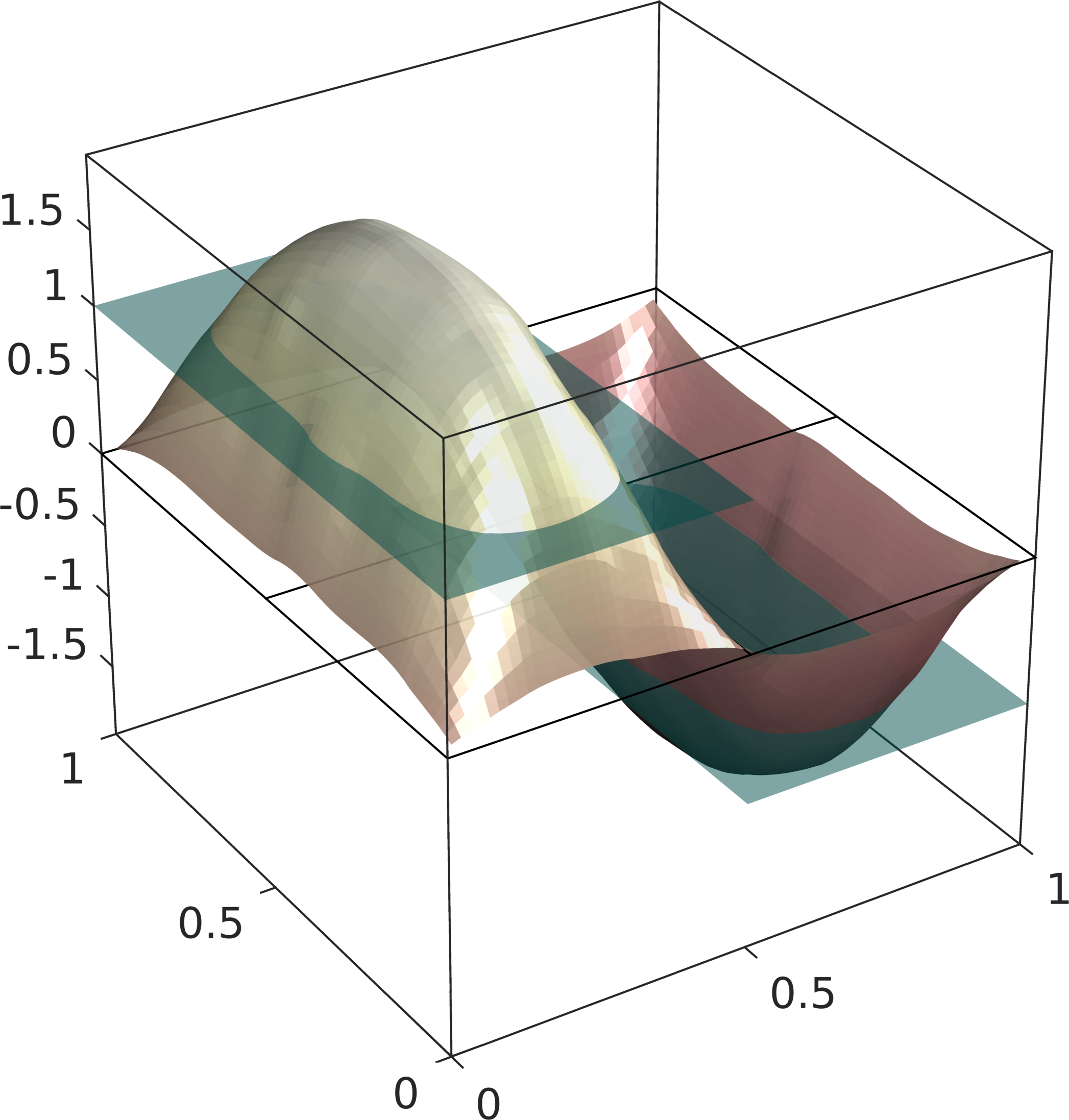}
	\includegraphics[height=0.118\textheight]{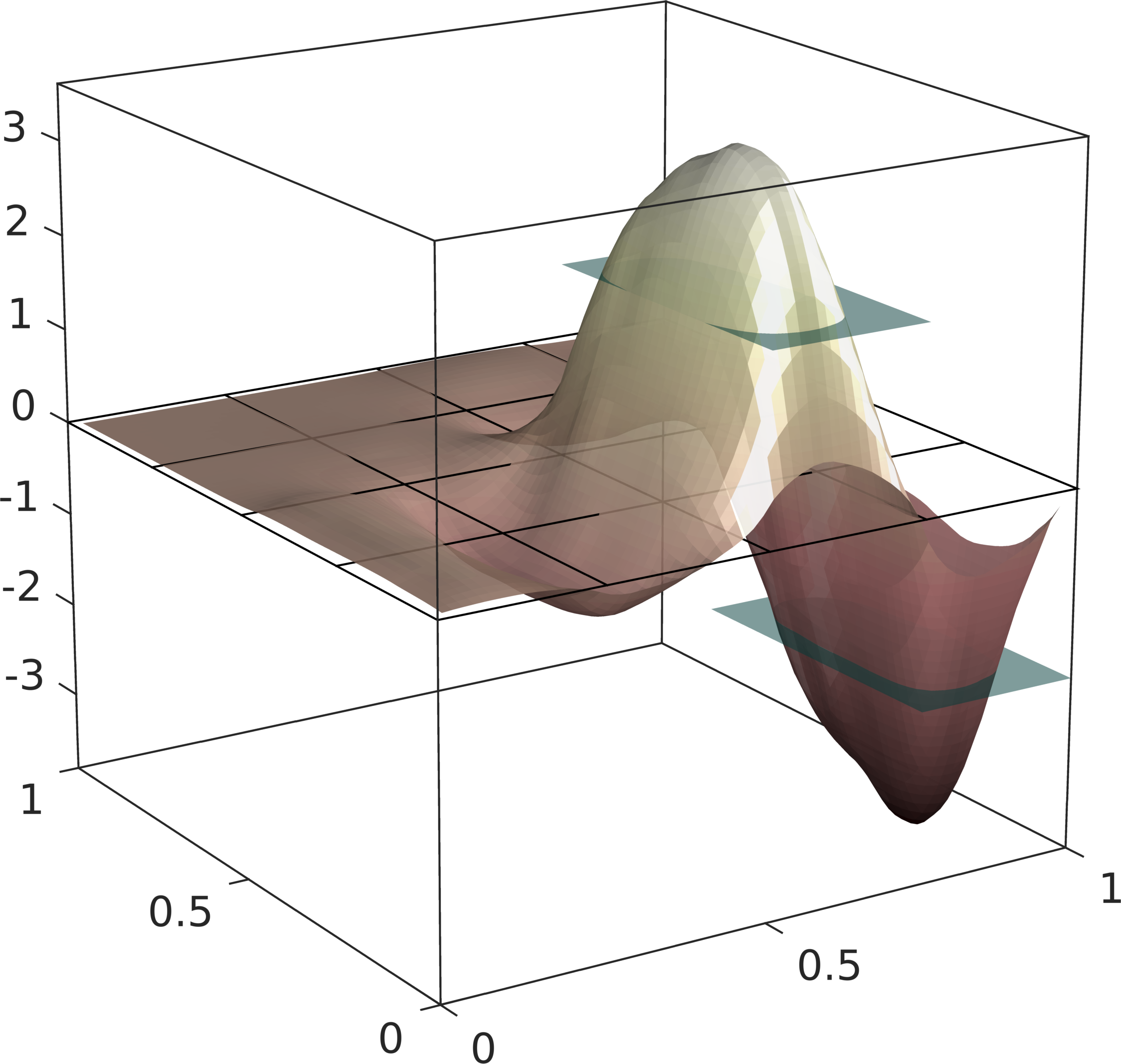}
	\includegraphics[height=0.118\textheight]{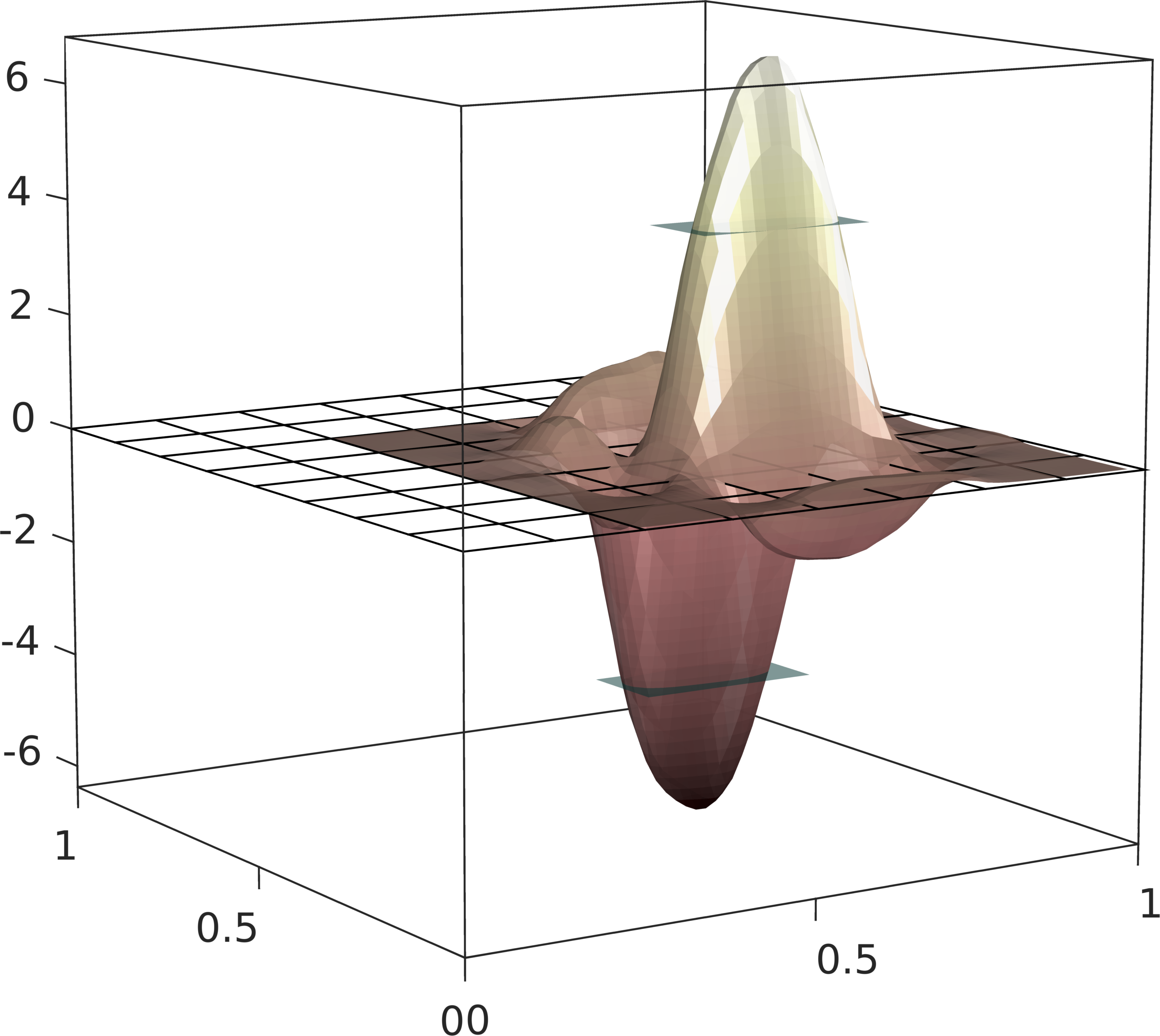}	\includegraphics[height=0.118\textheight]{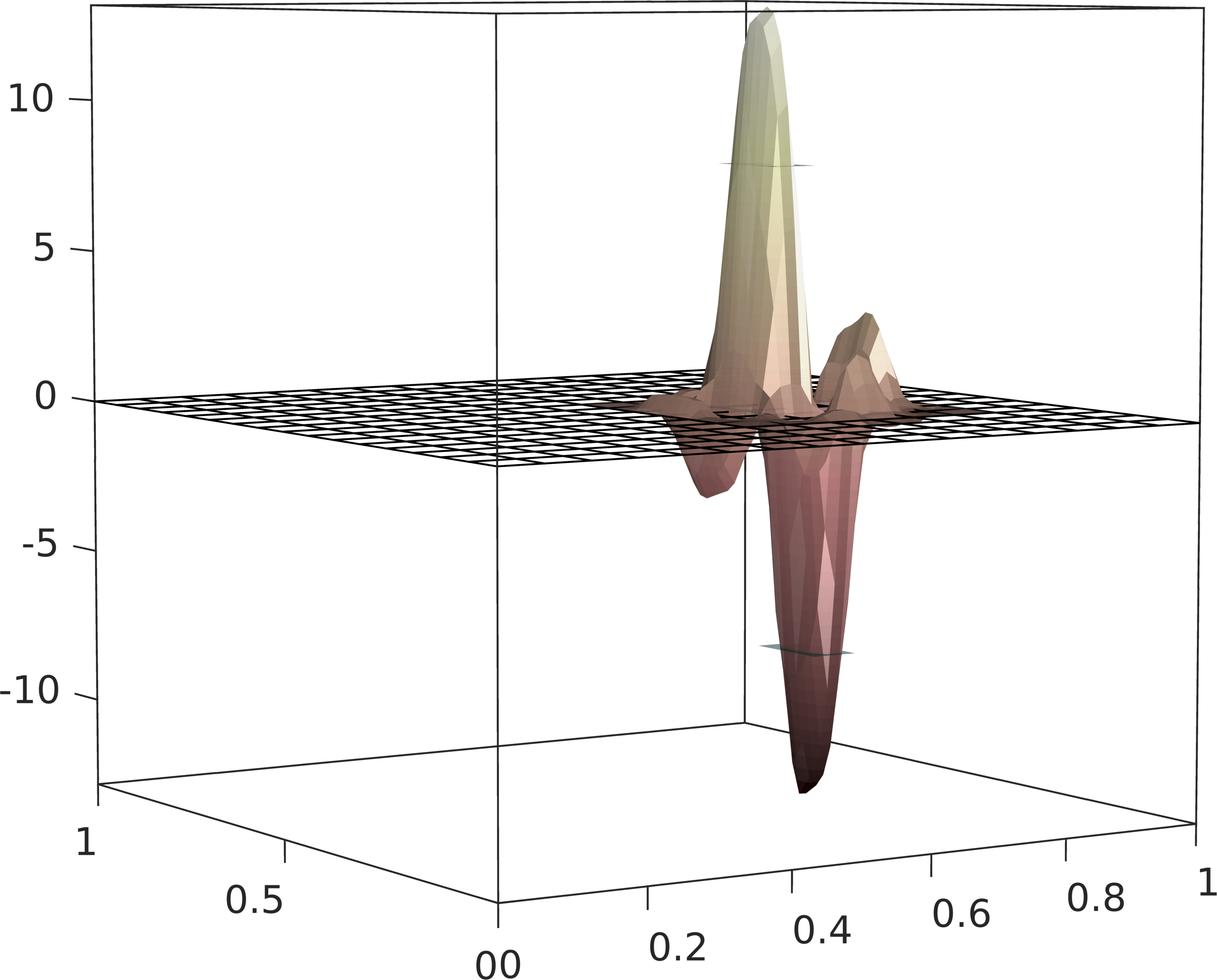}\\\caption{Realization of coefficient-adapted hierarchical decomposition in $1D$ (top row) and $2D$ (bottom row) based on an elliptic partial differential operator with random coefficient i.i.d. with respect to Cartesian grid of width $\varepsilon=2^{-6}$. Five levels from coarse (left) to fine (right). Green lines/surfaces represent classical Haar wavelets.}
	\label{fig:gamblet1d}
\end{figure}
Due to the lack of $V$-conformity, the basis $\HH$ is not suited for approximating the solution of model problem \eqref{e:modelweak} in a Galerkin approach. It will, however, serve as a companion of certain regularized hierarchical bases $\BBb(\omega)=\bigcup_{\ell=0}^L\BBb_\ell(\omega)\subset V$ to be defined below. The new bases are connected to $\HH$ (and to each other) via $L^2$-orthogonal projections $\Pi_\ell\colon V\to P^0(\TT_\ell)$ onto $\TT_\ell$-piecewise constant functions by
\begin{equation}\label{eq:BBb1}
\Pi_\ell\BBb_\ell(\omega)=\HH_\ell
\end{equation} 
for all $\ell = 0,1,\ldots,L$ and $\omega\in\Omega$.
Among the infinitely many possible choices, we define the elements of  $\BBb_\ell(\omega)$ by minimizing the energies $\tfrac{1}{2}\ab_\omega(\bullet,\bullet)$ in the closed affine space of preimages of $\Pi_\ell$ restricted to $V$, i.e., given $\phi\in\HH_\ell$ and $\omega\in\Omega$, we define 	$\bb_\phi(\omega)\in\BBb_\ell(\omega)$ by
\begin{align}\label{eq:defb}
	\bb_\phi(\omega):=\underset{v\in V}{\rm argmin}\;\tfrac{1}{2}\ab_\omega(v,v)\quad\text{subject to}\quad\Pi_\ell v = \phi.
\end{align}
This construction is strongly inspired by numerical homogenization where this sort of orthogonalization of scales in the energy space paved the way to a scheme that works with arbitrary rough coefficients beyond periodicity or scale separation \cite{MalP14,HenningPeterseim2013,Peterseim2016,KorPY18}. While most results in the context of this so-called localized orthogonal decomposition (LOD) are based on a conforming companion (the Faber basis), early works also addressed the possibility of using discontinuous companions \cite{doi:10.1137/120863162,Elfverson.Georgoulis.Mlqvist.ea:2012}. This dG version of LOD is very useful when taking the step from two levels or scales in numerical homogenization to actual multilevel decomposition. This was first shown in \cite{gamblets} where so-called gamblets are introduced; see also \cite{2017arXiv170602205S,MR3783084,MR3736719,OWHADI201799,MR3971243}. In particular, piecewise constants induce a natural hierarchical structure with nested kernels of local projection operators (here the $\Pi_\ell$) that is not easily achieved with $H^1$-conforming functions. The construction of the present paper coincides with the gamblet decomposition of \cite{gamblets} in the sense that the approximation spaces on all levels coincide in some idealized deterministic setting. However, our particular choice of basis is connected to the Haar-wavelets which decouples the definition and computation of the basis across levels. More importantly, our particular choice of basis is crucial in the context of the random problem at hand because it is exactly the link to the deterministic Haar basis that allows a meaningful interpretation of the averaged approximate solution operator. 
\smallskip

We shall express the mapping of bases encoded in \eqref{eq:BBb1}--\eqref{eq:defb} in terms of two concatenated linear operators. This will be useful for both analysis and actual computations. First, let $\tilde\Pi_\ell:L^2(D)\to V$ be such that 
\begin{equation}\label{e:tildephi1}
\Pi_\ell \circ \tilde\Pi_\ell = \id\quad\text{on }{\rm span}\HH_\ell
\end{equation}
In particular, this means that $\tilde\Pi_\ell$ maps any $\phi\in\HH_\ell$ to some  function that is admissible in the sense of the minimization problem \eqref{eq:defb}.
The operators $\tilde\Pi_\ell$ are easily constructed using non-negative bubble functions $\tilde\chi_T$ supported on an element $T\in\TT_\ell$ with $\Pi_\ell\tilde\chi_T=\chi_T$. Then 
$$\tilde\Pi_\ell v = \sum_{T\in\TT_\ell}(\Pi_\ell v)\vert_T\tilde\chi_T.$$
There is even locality in the sense of
\begin{equation}
	{\rm supp}\,\tilde\Pi_\ell\phi\subset {\rm supp}\,\phi\label{e:tildephi2}
\end{equation}
for all $\phi\in{\rm span}\HH_\ell$. The bubbles can be chosen such that, for some $C>0$,
\begin{equation}\label{e:tildephi3}
\|\tilde\Pi_\ell\phi\|_{H^{m}(D)}\leq C h^{-m} \|\phi\|_{L^2(D)}
\end{equation}	
holds for $m\in\{0,1\}$.

The second step involves $\ab_\omega$-orthogonal projections $\CCb_\ell(\omega)$ onto the closed subspaces 
\begin{equation}\label{eq:Well}
W_\ell:={\rm  kernel}(\Pi_\ell\vert_{V}) = {\rm  kernel}(\tilde\Pi_\ell\vert_{V}) 
\end{equation}
of $V$. Given any $u \in V$, define $\CCb_\ell(\omega)u\in W_\ell$ as the unique solution of the variational problem 
\begin{align}\label{eq:psi}
\ab_\omega(\CCb_\ell(\omega)u,v)=\ab_\omega(u,v)\quad\text{for all }v\in W_\ell.
\end{align}
With the two operators $\tilde\Pi_\ell$ and $\CCb_\ell$ we rewrite \eqref{eq:defb} as 
\begin{equation*}
\bb_\phi = (\id-\CCb_\ell)\tilde\Pi_\ell\phi
\end{equation*}
for all $\phi\in \HH_\ell$ and $\ell=0,1,\ldots,L$. Actually, for any $\omega\in\Omega$, $(\id-\CCb_\ell(\omega))\tilde\Pi_\ell$ defines a bijection from $\HH$ to $\BBb(\omega)$ with left inverse $\Pi_\ell$. 

While the $L^2$-orthogonality of the Haar basis is not preserved under these mappings, we have achieved $\ab$-orthogonality between the levels of the hierarchies as shown in the following lemma.
\begin{lemma}[$\ab$-orthogonality and scaling of $\BB$]\label{lem:bbasis}
	Any two functions $b_k\in \BB_k(\omega)$ and $b_\ell\in \BB_\ell(\omega)$ with $k\neq\ell$ satisfy $$\ab_\omega(b_k,b_\ell)=0.$$
	Moreover, 
	\begin{align}\label{eq:normequiv}
C^{-1}\|\phi_k\|_{L^2(D)}\leq C^{-1}\|b_k\|_{L^2(D)}\leq h_k \enorm{b_k}_{\omega} \leq C \|\phi_k\|_{L^2(D)}
	\end{align} 
	with some generic constant $C>0$ independent of the mesh sizes and the event. 
\end{lemma}
\begin{proof}
	Since \begin{equation}\label{e:pik1}\Pi_k(\id-\CCb_\ell(\omega))\tilde\Pi_\ell\HH_\ell=\Pi_k\Pi_\ell(\id-\CCb_\ell(\omega))\tilde\Pi_\ell\HH_\ell=\Pi_k\HH_\ell=\{0\}
	\end{equation} whenever $k<\ell$, we have that 
	$$\BBb_\ell(\omega)\subset W_k.$$
	This and the orthogonality 
	$$\ab_\omega(\BBb_k(\omega),W_k)=0$$
	from \eqref{eq:psi}	proves the (block-)orthogonality of the bases. 
	The scaling follows from $\Pi_{k-1}\BBb_k=\{0\}$ (which is a special instance of \eqref{e:pik1}), the Poincar\'e inequality, \eqref{e:tildephi3}, and the construction. More precisely, 
	\begin{align}\label{e:scaling}
	\begin{split}
	\|\phi_k\|_{L^2(D)}&=\|\Pi_k b_k\|_{L^2(D)}\leq\|b_k\|_{L^2(D)}=\|(1-\Pi_{k-1})b_k\|_{L^2(D)}\lesssim h_k\enorm{b_k}_{\omega}\\& = h_k\enorm{(1-\CCb(\omega))\tilde\Pi_k\phi_k}_{\omega} \leq h_k \enorm{\tilde\Pi_k\phi_k}_{\omega}\lesssim \|\phi_k\|_{L^2(D)}.
	\end{split}
	\end{align}
	This concludes the proof.
\end{proof}
We shall emphasize that, in general, the basis elements $\bb_\phi(\omega)$ have global support in $D$. However, their moduli decay exponentially away from ${\rm supp}\,\phi$ in scales of $h_\ell$,
\begin{equation}\label{e:decay}
\|\bb_\phi(\omega)\|_{H^1(D\setminus B_R({\rm supp}\phi))}\leq C e^{-cR/h_\ell}\|\bb_\phi(\omega)\|_{H^1(D)}
\end{equation}
with some generic constants $c,C>0$ that solely depend on the contrast $\gamma_{\operatorname{max}}/\gamma_{\operatorname{min}}$ and the shape regularity of the mesh $\TT_\ell$ (and thus on $\TT_0$) but not on the mesh size. This is a well established result of numerical homogenization since \cite{MalP14} and valid in many different settings (see \cite{Peterseim2016} and references therein). Here, we will provide some elements of a more recent constructive proof of the decay that provides local approximations by the theory of preconditioned iterative solvers \cite{KorPY18} which in turn is based on \cite{KorY16}.  

We start with introducing an overlapping decomposition of $D$ that we will later use to define the local preconditioner. Let the level $\ell\in\{0,1,\ldots,L\}$ and the event $\omega\in\Omega$ be arbitrary but fixed. For any element of the mesh, define the patch 
$$D_T:=\bigcup\{K\in\TT_\ell\;|\;\bar K\cap \bar T\neq\emptyset\}$$
and a corresponding local subspace
\[
V_T
:= \big\{ v\in V\ |\ v=0 \text{ in }  D\setminus D_T  \big\}\subset V.
\]
Note that $V_T$ is equal to $H^1_0(D_T)$ up to extension by zero outside of $D_T$.  Let $\lambda_T$, $T\in\TT_\ell$, be a partition of unity
with $\operatorname{supp}\lambda_T\subset D_T$ and $\|\lambda_T\|_{W^{m,\infty}(D)}\lesssim h_\ell^{-m}$, $m=0,1$.
Under the complementary projection $(\id-\tilde\Pi_\ell)$ these subspaces are turned into subspaces
\[
W_T 
:= (\id-\tilde\Pi_\ell)V_T = \big\{ v\in W_\ell\ |\ v=0 \text{ in } D\setminus D_T  \big\}
\]
of $W_\ell$. For each $T\in \TT_\ell$ we define the corresponding $\ab_\omega$-orthogonal projection $\PPb_T(\omega)\colon V\to W_T\subset W_\ell\subset V$ by the variational problem
\[
\ab_\omega(\PPb_T(\omega) u, w) = \ab_\omega(u,w)\quad\text{for all }w\in W_T.
\]
The sum of these local Ritz projections  
\begin{align}
\label{def_calP}
\PPb_\ell(\omega) 
:= \sum\nolimits_{T\in\TT_\ell} \PPb_T(\omega)
\end{align}
defines a bounded linear operator from $V$ to $W_\ell$ that can be seen as a preconditioned version of the correction operator $\CCb_\ell(\omega)$. The operator $\PPb_\ell(\omega)$ is quasi-local with respect to the mesh $\TT_\ell$ since information can only propagate over distances of order $h_\ell$ each time $\PPb_\ell(\omega)$ is applied. 

The remaining part of this section aims to show that the preconditioned operators $\PPb_\ell(\omega)$ serve well within iterative solvers for linear equations. Following the abstract theory for subspace correction or additive Schwarz methods for operator equations \cite{KorY16} (see also \cite{Xu92,yserentant_1993} for the matrix case) we need to verify that the energy norm of a function $u\in V$ can be bounded in terms of the sum of local contributions from $V_T$ and, for one specific decomposition, we need a reverse estimate.
\begin{lemma}\label{lem:K1K2}
	For every decomposition $u = \sum_{T\in\TT_\ell} u_T$ of $u\in W_\ell$ with $u_T \in W_T$
	we have 
	\[
	\|\nabla u\|^2_{L^2(D)}
	\le K_2\,\sum\nolimits_{z\in\TT_\ell} \|\nabla u_T\|^2_{L^2(D)}
	\]
	with constant $K_2>0$ depending only on the shape regularity of $\TT_\ell$ (and thus on $\TT_0$). With the partition of unity functions $\lambda_T$ associated with the elements $T\in\TT_\ell$, the one decomposition $\sum_{T\in \TT_\ell} u_T = u$
	with $u_T := (1-\tilde\Pi_\ell)(\lambda_T u)\in W_T$ for  $T\in \TT_\ell$ satisfies
	\[
	\sum\nolimits_{T\in\TT_\ell} \|\nabla u_T\|_{L^2(D)}^2\leq K_1\, \|\nabla u \|_{L^2(D)}^2 
	\]
	with constant $K_1>0$ that only depends on the shape regularity of $\TT_\ell$ and the contrast $\gamma_{\operatorname{max}}/\gamma_{\operatorname{min}}$.
\end{lemma}
\begin{proof}
With $K_2$ the maximum number of elements of $\TT_\ell$ covered by one patch $D_T$ for $T\in \TT_\ell$, we can estimate on a single element $T'$,  
	\[
	\|\nabla u\|_{L^2(T')}^2 
	= \| \sum_{T\in\TT_\ell} \nabla u_T \|_{L^2(T')}^2
	\le K_2\,\sum_{T\in\TT_\ell} \| \nabla u_T \|_{L^2(T')}^2.
	\]
Due to shape regularity of $\TT_\ell$, $K_2$ is independent of $h_\ell$. A summation over all $T'$ yields the first inequality. The second one follows from the $H^1$-stability of $\tilde\Pi_\ell$ on $W_\ell$, the product rule, \eqref{e:tildephi3}, and the Poincar\'e inequality. For further details, we refer to \cite[Lemma 3.1]{KorPY18} where these results are proved in detail in a very similar setting.
\end{proof}
Lemma~\ref{lem:K1K2} implies that 
\begin{equation}
1/K_1 \ab_\omega(v,v)\leq \ab_\omega(\PPb_\ell(\omega)v,v)\leq K_2\ab_\omega(v,v)
\end{equation}
holds for functions $v$ in the kernel $W_\ell$ of $\Pi_\ell\vert_V$ and any $\omega\in\Omega$ (cf. \cite[Eq. (3.11)]{KorPY18}). Following the construction of \cite{KorY16,KorPY18} there exists a localized linear approximation $\CCb_\ell^\delta(\omega)$ based on $\OO(\log(1/\delta))$ steps of some linear iterative solver applied to the preconditioned corrector problems \cite[Eqns. (3.8) or (3.18)]{KorPY18} such that 
\begin{equation}\label{eq:decay}
\|\nabla(\CCb_\ell(\omega)u-\CCb_\ell^\delta(\omega)u)\|_{L^2(D)}\leq \delta\|\nabla\CCb_\ell(\omega)u\|_{L^2(D)};
\end{equation}
see \cite[Lemma 3.2]{KorPY18}. With the approximate correctors, we can define modified (localized) bases
$$\BB^\delta(\omega):=\bigcup_{\ell=0}^L \BB_\ell^\delta(\omega):=\bigcup_{\ell=0}^L\set{\bb_\phi^\delta(\omega)}{\phi\in\HH_\ell},$$ 
where
\begin{align*}
	\bb_\phi^\delta(\omega):=(\id-\CCb_\ell^\delta(\omega))\tilde\Pi_\ell\phi
\end{align*}
for $\phi\in\HH_\ell$.
The previous discussion shows that there exist constants 
$C_1,C_2>0$ that only depend on the shape regularity of the meshes $\TT_\ell$ and the contrast $\gamma_{\operatorname{max}}/\gamma_{\operatorname{min}}$ of the coefficients such that 
\begin{align}\label{eq:bbapprox}
\enorm{\bb_\phi(\omega)-\bb_\phi^\delta(\omega)}_\omega\leq C_1 \delta\enorm{\bb_\phi(\omega)}_\omega
\end{align}
while 
\begin{equation}\label{eq:bsupport}
{\rm supp}\;\bb_\phi^\delta(\omega)\subset \set{x\in D}{{\rm dist}(x,{\rm supp}\;\phi)\leq C_2|\log(\delta)|h_\ell}.
\end{equation}
Later on we will typically use normalized bases. 
	Since 
	\begin{align}\label{eq:normequiv2}
	(1-C_1\delta)\enorm{\bb_\phi(\omega)}_\omega\leq\enorm{\bb_\phi^\delta(\omega)}_\omega \leq (1+C_1\delta)\enorm{\bb_\phi(\omega)}_\omega
	\end{align}
	by \eqref{eq:bbapprox}, the normalization of the localized bases is meaningful whenever if $\delta< 1/C_1$. Normalization does not affect the local supports \eqref{eq:bsupport} and the approximation property \eqref{eq:bbapprox} is preserved in the following sense, 
\begin{align}\label{eq:bbapproxnorm}
\begin{split}
\enorm{\frac{\bb_\phi(\omega)}{\enorm{\bb_\phi(\omega)}_\omega}-\frac{\bb_\phi^\delta(\omega)}{\enorm{\bb_\phi^\delta(\omega)}_\omega}}_\omega
&\leq \frac{\enorm{\bb_\phi(\omega)-\bb_\phi^\delta(\omega)}_\omega}{\enorm{\bb_\phi(\omega)}_\omega} +
\frac{\enorm{\bb_\phi(\omega)}_\omega-\enorm{\bb_\phi^\delta(\omega)}_\omega}{\enorm{\bb_\phi^\delta(\omega)}_\omega}
\\
&\leq\delta C_1 +\frac{C_1\delta}{1-C_1\delta}\leq 3 \delta C_1.
\end{split}
\end{align}
for any $\delta\leq 1/(2C_1)$.

\section{Sparse stiffness matrices}\label{sec:hierarchical}
With the localized bases of the previous section, we can now study the sparsity of corresponding stiffness matrices and their inverses. We define the level function $\level(\cdot)$ according to the Haar basis by $\level(\bb)=\level(\bb^\delta)=\level(\phi)=\ell$ for $\bb=\bb_\phi\in \BB_\ell(\omega)$, $b^\delta=b^\delta_\phi\in \BB_\ell^\delta(\omega)$ and $\phi\in\HH_\ell$.
We order the basis functions in $\BB$, $\BB^\delta$, and $\HH$ such that $\level$ is monotonically increasing in the index running from $1$ to  $N:=\#\BB=\#\BB^\delta =\#\HH$. With this convention, we may also write  $\level(i):=\level(\bb_i)=\level(\bb_i^\delta)=\level(\phi_i)$ for all $i=1,\ldots,N$.
Moreover, we define a (semi-)metric $d(\cdot,\cdot)$ on $\{1,\ldots,N\}$ by
\begin{align*}
	d(i,j):=
		\frac{{\rm dist}({\rm mid}(\phi_i),{\rm mid}(\phi_j))}{h_{\min\{\level(i),\level(j)\}}},
\end{align*}
where ${\rm mid}(w)$ defines the barycenter of ${\rm supp}(w)$.

Define the stiffness matrices $\Sb(\omega)\in \R^{N\times N}$ associated with the bases $\BBb(\omega)$ by
$$\Sb(\omega)_{ij}:= \ab_\omega\biggl(\frac{\bb_j(\omega)}{\enorm{\bb_j(\omega)}_\omega},\frac{\bb_i}{\enorm{\bb_i(\omega)}_\omega}\biggr).$$
The orthogonality of the bases $\BBb$ motivates the approximation of the stiffness matrices by block-diagonal ones even after localization. 
Given $1/C_1>\delta>0$, define the block-diagonal stiffness matrices $\Sb^\delta(\omega)\in \R^{N\times N}$ by
\begin{align*}
\Sb^\delta(\omega)_{ij}:=\begin{cases}
\ab_\omega\bigl(\frac{\bb_j^\delta(\omega)}{\enorm{\bb_j^\delta(\omega)}_\omega},\frac{\bb_i^\delta(\omega)}{\enorm{\bb_i^\delta(\omega)}_\omega}\bigr)&\text{for }\level(i)=\level(j),\\
0&\text{else.}
\end{cases}
\end{align*}
In the following, we use the spectral norm $\norm{\cdot}{2}$, i.e., the matrix norm induced by the Euclidean norm.
\begin{lemma}\label{lem:stiffness}
There exists a constant $C>0$ that depends only on $D$ and the shape regularity of $\TT_0$ and the contrast $\gamma_{\operatorname{max}}/\gamma_{\operatorname{min}}$ such that, for any $\omega\in\Omega$ and for all $\delta\leq 1/(2C_1)$,  $$\norm{\Sb(\omega)-\Sb^\delta(\omega)}{2}\leq C\delta.$$
Moreover, there exists a constant $\zeta>0$ which depends only on $D$ such that
	\begin{align}\label{eq:banded}
		 d(i,j)>\zeta(|\log(\delta)|+1) \text{ or }\level(i)\neq \level(j)\quad\implies\quad \Sb_{ij}^\delta(\omega)=0,
	\end{align}
in particular, the number of nonzero entries $\operatorname{nnz}(\Sb^\delta(\omega))\lesssim N(1+|\log\delta|)^{d}$ is bounded uniformly in $\omega$.
\end{lemma}
\begin{proof}
	The sparsity of the diagonal blocks follows from \eqref{eq:bsupport}. For the proof of the error bound, define 
	\begin{align*}
	\widetilde \Sb^\delta(\omega)_{ij}:=\begin{cases}
	\ab_\omega\bigl(\frac{\bb_j(\omega)}{\enorm{\bb_j(\omega)}_\omega},\frac{\bb_i^\delta(\omega)}{\enorm{\bb_i^\delta(\omega)}_\omega}\bigr)&\text{for }\level(i)=\level(j),\\
	0&\text{else.}
	\end{cases}
	\end{align*}
 Since $	|\Sb_{ij}(\omega)-\widetilde \Sb_{ij}^\delta(\omega)|=0$ whenever $\level(i)\neq \level(j)$ it suffices to bound the errors related to the diagonal blocks indexed by $\ell=1,2,\ldots,L.$
	We have for any vectors $x,y\in\mathbb{R}^{\#\BB_\ell^\delta}$ that
	\begin{multline*}
	|x\cdot(\Sb_\ell(\omega)-\widetilde \Sb_\ell^\delta(\omega))y| \\=\biggl|\sum_{\level(i)=\ell}\sum_{\level(j)=\ell}x_iy_i\ab_\omega\Bigl(\frac{\bb_i(\omega)}{\enorm{\bb_i(\omega)}_\omega},
	\frac{\bb_i^\delta(\omega)}{\enorm{\bb_i^\delta(\omega)}_\omega}-\frac{\bb_i(\omega)}{\enorm{\bb_i(\omega)}_\omega}\Bigr)\biggr|
	\lesssim \delta \norm{x}{\ell_2}\norm{y}{\ell_2}
	\end{multline*}
	by \eqref{eq:normequiv2}. 
	The same arguments show $|x\cdot(\Sb_\ell^\delta(\omega)-\widetilde \Sb_\ell^\delta(\omega))y|\lesssim \delta\norm{x}{\ell_2}\norm{y}{\ell_2}$ and the triangle inequality readily proves the assertion.
\end{proof}

\begin{lemma}\label{lem:riesz}For any $\omega\in\Omega$ the normalized set $\BB=\BBb(\omega)$ or $\BB=\BBb^\delta(\omega)$ (with $\delta \lesssim 1/L$ sufficiently small) is a Riesz bases in the sense that
	\begin{align}\label{eq:riesz}
C^{-1}\sum_{b\in\BB}\alpha_{b}^2\leq	\biggl\|\;\sum_{b\in\BB}\alpha_{b} \frac{b}{\enorm{b}_\omega}\biggr\|_{H^1(D)}^2\leq C\sum_{b\in\BB}\alpha_{b}^2
	\end{align}	
	holds with some constant $C>0$ which depends only on $D$ and the shape regularity of $\TT_0$ and the contrast $\gamma_{\operatorname{max}}/\gamma_{\operatorname{min}}$. This immediately implies that $\Sb(\omega)$ and $\Sb^\delta(\omega)$ are uniformly well conditioned.
\end{lemma}
\begin{proof}		
	Since $\norm{\cdot}{H^1(D)}$ and $\enorm{\cdot}_{\omega}$  are equivalent uniformly in $\omega$ and the basis $\BBb(\omega)$ is $\ab_\omega(\cdot,\cdot)$-orthogonal 
	across the levels, it suffices to consider one level $k\in\{1,\ldots,L\}$ in the case $\BB=\BBb(\omega)$. The $L^2(D)$-orthogonality of the Haar basis and the construction of $\BB$ implies 
	\begin{multline*}
	\sum_{b\in\BB_k}\alpha_{b}^2 =
	\biggl\|	
	\sum_{b\in\BB_k}\alpha_{b}\frac{\phi_{b}}{\|\phi_{b}\|_{L^2(D)}}\biggr\|_{L^2(D)}^2
	=
	\biggl\|\Pi_k	\sum_{b\in\BB_k}\alpha_{b}\frac{b}{\|\phi_{b}\|_{L^2(D)}}\biggr\|_{L^2(D)}^2\\\leq
	\biggl\|(1-\Pi_{k-1})\sum_{b\in\BB_k}\alpha_{b}\frac{b}{\|\phi_{b}\|_{L^2(D)}}\biggr\|_{L^2(D)}^2\lesssim 	\biggl\|\sum_{b\in\BB_k}\alpha_{b}\frac{b}{\enorm{b}_\omega}\biggr\|_{H^1(D)}^2,
	\end{multline*} 
	where the last estimate follows from the Poincar\'e inequality and \eqref{eq:normequiv}. For the proof of the converse direction, the construction of $\BB$ and boundedness of $\CCb_k$ show
	\begin{multline*} 	\biggl\|\sum_{b\in\BB_k}\alpha_{b}\frac{b}{\enorm{b}_\omega}\biggr\|_{H^1(D)}^2=
	\biggl\|(\id-\CCb_k(\omega))\tilde\Pi_k\sum_{b\in\BB_k}\alpha_{b}\frac{b}{\enorm{b}_\omega}\biggr\|_{H^1(D)}^2\lesssim 	\biggl\|\tilde\Pi_k\sum_{b\in\BB_k}\alpha_{b}\frac{\phi_{b}}{\enorm{b}_\omega}\biggr\|_{H^1(D)}^2\\
	\lesssim 	\biggl\|\sum_{b\in\BB_k}\alpha_{b}\frac{\phi_{b}}{\|\phi_{b}\|_{L^2(D)}}\biggr\|_{L^2(D)}^2
	=	\sum_{b\in\BB_k}\alpha_{b}^2,
	\end{multline*}
		where the second inequality follows from the inverse inequality 	\eqref{e:tildephi3} and \eqref{eq:normequiv}.
		
The result for $\BB=\BBb^\delta(\omega)$ is slightly more involved as the $\ab_\omega(\cdot,\cdot)$-orthogonal 
across the levels is lost. In a first step, Lemma~\ref{lem:stiffness} and the equivalence of $\enorm{\cdot}_\omega$ and $\norm{\cdot}{H^1(D)}$ imply
	\begin{align*}
	 \biggl\|\;\sum_{b^\delta\in\BBb^\delta_\ell(\omega)}\alpha_{b^\delta} \frac{b^\delta}{\enorm{b^\delta}_\omega}\biggr\|_{H^1(D)}^2\simeq \Sb_\ell^\delta(\omega)\alpha\cdot \alpha\simeq 
	 \Sb_\ell(\omega)\alpha\cdot \alpha \pm C\delta\norm{\alpha}{\ell_2}^2\simeq (1\pm C\delta)\norm{\alpha}{\ell_2}^2
	\end{align*}
with the constant $C$ from Lemma~\ref{lem:stiffness} and $\delta\leq 1/(2C)$. The second step concerns the quantification of non-orthogonality. The estimate~\eqref{eq:decay} and the norm equivalence $\enorm{\cdot}_\omega\simeq \norm{\cdot}{H^1(D)}$ imply
	\begin{align*}
	 \enorm{(\CCb_\ell^\delta-\CCb_\ell)v}_\omega \lesssim \delta \enorm{v}_\omega\quad\text{for all }v\in H^1_0(D).
	\end{align*}
    Consequently, we obtain as in the proof of Lemma~\ref{lem:stiffness},  for some $1\leq k\neq\ell\leq L$
	\begin{align*}
	 \ab_\omega&\Big(\sum_{b^\delta\in\BBb^\delta_\ell(\omega)}\alpha_{b^\delta} \frac{b^\delta}{\enorm{b^\delta}_\omega},
	 \sum_{b\in\BBb_k(\omega)}\beta_{b} \frac{b}{\enorm{b}_\omega}\Big)\\
	&=
	 \ab_\omega\Big((\CCb_\ell^\delta-\CCb_\ell)\tilde\Pi_\ell\sum_{b_\phi^\delta\in\BBb^\delta_\ell(\omega)}\beta_{b_\phi^\delta} \frac{\phi}{\enorm{b_\phi^\delta}_\omega},
	 \sum_{b\in\BBb_k(\omega)}\alpha_{b} \frac{b}{\enorm{b}_\omega}\Big)\\
	 &\lesssim \delta\norm{\alpha}{\ell_2}\norm{\beta}{\ell_2},
	\end{align*}
where we used the orthogonality across levels and the stability of $\BBb$. Symmetry of the argument concludes $\ab_\omega\Big(\sum_{b^\delta\in\BBb^\delta_\ell(\omega)}\alpha_{b^\delta} \frac{b^\delta}{\enorm{b^\delta}_\omega},
	 \sum_{b\in\BBb_k(\omega)}\beta_{b} \frac{b}{\enorm{b}_\omega}\Big)\lesssim \delta\norm{\alpha}{\ell_2}\norm{\beta}{\ell_2}$ and we find
	\begin{align*}
	  \biggl\|\;\sum_{b^\delta\in\BBb^\delta(\omega)}\alpha_{b^\delta} \frac{b^\delta}{\enorm{b^\delta}_\omega}\biggr\|_{H^1(D)}^2&\simeq 
	  \sum_{\ell=1}^L (1\pm \delta)\norm{\alpha|_{\BBb^\delta_\ell(\omega)}}{\ell_2}^2 \pm \delta\sum_{i,j=1\atop i\neq j}^L\norm{\alpha|_{\BBb^\delta_i(\omega)}}{\ell_2}\norm{\alpha|_{\BBb^\delta_j(\omega)}}{\ell_2}\\
	  &\simeq (1 \pm C\delta L) \norm{\alpha}{\ell_2}^2.
	\end{align*}
	for sufficiently small $\delta\leq1/(2CL)$. This concludes the proof.
\end{proof}

\section{Basis transformations}\label{sec:transform}
This section analyzes the properties of a certain matrix representation of the $L^2(D)$-orthogonal projections $\Pi_\ell\colon L^2(D)\to {\rm span}(\bigcup_{j=1}^\ell\HH_\ell)$ for $\ell=1,\ldots,L$. 
Given $\omega\in\Omega$, define the matrix $\Tb(\omega)\in\R^{N\times N}$ by
 $$\Tb_{ij}(\omega):=\frac{(\bb_j(\omega),\phi_i)_{L^2(D)}}{\enorm{\bb_j(\omega)}_\omega\norm{\phi_i}{L^2(D)}^2}.$$
Given some $v=\sum_{i=1}^N \alpha_i \frac{\bb_i}{\enorm{\bb_i}_\omega}$ with $\Pi_L v = \sum_{i=1}^N \beta_i \frac{\phi_i}{\norm{\phi_i}{L^2(D)}}$.
Then, by definition
\begin{align*}
\beta_i = \frac{(\Pi_L v , \phi_i)}{\norm{\phi_i}{L^2(D)}^2} = \sum_{j=1}^N \alpha_j  \Tb_{ij}=(\Tb\alpha)_i,
\end{align*}
i.e., $\beta = \Tb(\omega)\alpha$. 
Given $\delta>0$, a truncated approximation $\Tb^\delta(\omega)$ of $\Tb(\omega)$ is defined by
\begin{equation*}
\Tb^\delta_{ij}(\omega):=\begin{cases}\frac{(\bb_j^\delta(\omega),\phi_i)_{L^2(D)}}{\enorm{\bb^\delta_j(\omega)}_{\omega}\norm{\phi_i}{L^2(D)}^2}& \text{if } \level(j)\leq \level(i),\\
0&\text{otherwise},
\end{cases} 
\end{equation*}
for any $i,j\in\{1,\ldots,N\}$. $\Tb^\delta(\omega)$ is a sparse lower block-triangular matrix and the next lemma shows that the error of truncation is at most proportional to $\delta$. To explore the block-structure of matrices we shall introduce the following notation first. For any matrix $K\in\R^{N\times N}$, we define sub-blocks $K_{(k,\ell)}\in \R^{\#\HH_\ell\times \#\HH_k}$ according to the level structure by
\begin{align*}
 K_{(k,\ell)}:=K|_{\set{(i,j)}{\level(i)=k,\,\level(j)=\ell}}.
\end{align*}
Thus, we may write
\begin{align*}
 K=\begin{pmatrix}
    K_{(0,0)} & K_{(0,1)} & \cdots & K_{(0,L)}\\
    K_{(1,0)} & K_{(1,1)} & \cdots  & K_{(1,L)}\\
    \vdots & \vdots &\ddots & \vdots\\
    K_{(L,0)} & K_{(L,1)} & \cdots & K_{(L,L)}
   \end{pmatrix}.
\end{align*}

\begin{lemma}\label{lem:basistrans}
For $\delta>0$ as in Lemma~\ref{lem:riesz}, there holds
$$\norm{\Tb(\omega)-\Tb^\delta(\omega)}{2}\leq CL\delta$$
and, for $0\leq \ell\leq k\leq L$, there holds
\begin{align}\label{eq:Tdecay}
\norm{\Tb^\delta(\omega)_{(k,\ell)}}{2}\leq Ch_k.
\end{align} 
Moreover, $\Tb^\delta$ is lower block-triangular with sparse blocks, more precisely,
	\begin{align*}
		\big(\level(j)\geq \level(i)\text{ and }i\neq j\big)\text{ or } d(i,j)>\zeta(1+|\log(\delta)|)\quad
		\implies\quad \Tb^\delta_{ij}=0,
	\end{align*}
where $\zeta>0$ is the bandwidth from Lemma~\ref{lem:stiffness}. 
The number of nonzero entries per block is bounded by $\operatorname{nnz}(\Tb^\delta(\omega)_{(k,\ell)})
\lesssim \#\HH_k(1+|\log\delta|)^{d}$.
The constant $C>0$ depends only on $D$, the shape regularity of $\TT_0$ and the contrast 
$\gamma_{\operatorname{max}}/\gamma_{\operatorname{min}}$.
\end{lemma}
\begin{proof}
We see immediately $\Tb_{ij}(\omega)=0$ for all $\level(j)\geq \level(i)$ and $i\neq j$ since $$(\bb_j(\omega),\phi_i)_{L^2(D)}=(\Pi_{\level(\phi_i)} \bb_j(\omega), \phi_i)_{L^2(D)} = (\phi_j,\phi_i)_{L^2(D)}=0.$$ Since ${\rm supp}(\bb_i^\delta(\omega))\cap{\rm supp}(\phi_j)=\emptyset$ as soon as $d(i,j)\gtrsim |\log(\delta)|$, there is some $\zeta>0$ which depends only on $D$ such that  $\Tb^\delta(\omega)_{ij}=0$  for all $d(i,j)> \zeta(1+|\log(\delta)|)$.

For any vectors $x\in\mathbb{R}^{\#\BB_k^\delta}$ and $y\in\mathbb{R}^{\#\BB_\ell^\delta}$, we have
\begin{multline*}
	x\cdot(\Tb_{(k,\ell)}(\omega)-\Tb^\delta_{(k,\ell)}(\omega))y\\=\sum_{\level(i)=k}\sum_{\level(j)=\ell}x_iy_j\biggl(\frac{\phi_i}{\norm{\phi_i}{L^2(D)}}, 
	\frac{\bb_j^\delta(\omega)}{\enorm{\bb_j^\delta(\omega)}_\omega}-\frac{\bb_j(\omega)}{\enorm{\bb_j(\omega)}_\omega}\biggr)_{L^2(D)}\lesssim\delta\norm{x}{\ell_2}\norm{y}{\ell_2}
\end{multline*}
by Friedrichs' inequality and \eqref{eq:bbapproxnorm}.
This implies $\norm{\Tb(\omega)_{(\ell,k)}-\Tb^\delta(\omega)_{(\ell,k)}}{2}\lesssim \delta$.
Summing up over the levels proves $\norm{\Tb(\omega)-\Tb^\delta(\omega)}{2}\lesssim L\delta$.
 
To see~\eqref{eq:Tdecay}, note that $w:=\sum_{\phi_i\in\HH_k} \alpha_i \phi_i$ and $b:=\sum_{\bb_j(\omega)\in \BB_\ell}\beta_j\bb^\delta_j(\omega)$ satisfy 
	\begin{align*}
 \alpha^T \Tb^\delta(\omega)\beta&=(w , b)_{L^2(D)}=((1-\Pi_k)w , b)_{L^2(D)}=(w ,(1-\Pi_k) b)_{L^2(D)}\\&\lesssim h_k\norm{w}{L^2(D)}\enorm{b}_{\omega}\lesssim h_k\norm{\alpha}{\ell_2}\norm{\beta}{\ell_2}
\end{align*}
by Lemma~\ref{lem:riesz}. This concludes the proof.
\end{proof}

\section{Inverse stiffness matrices and averaging}\label{s:inverse}
This section proves that the inverse of the stiffness matrix $\Sb^\delta(\omega)$ (w.r.t. the coefficient adapted bases $\BBb^\delta(\omega)$) defined in the previous section can be efficiently approximated by a sparse matrix. 
One possibility to compute an approximate inverse of the matrix $\Sb^\delta(\omega)$ is to apply the conjugate gradient method (CG) to the matrix with unit vectors $e_i\in\R^N$
 as right-hand sides. 
 The sparsity pattern from Lemma~\ref{lem:stiffness} shows that one matrix-vector product with $\Sb^\delta e_i$ increases the number of non-zero entries to $\#\set{1\leq j\leq N}{d(i,j)\lesssim 1+|\log(\delta)|}$.
 Thus, after $k\in\N$ iterations of the CG method, the resulting vector has about $\#\set{1\leq j\leq N}{d(i,j)\lesssim k(1+|\log(\delta)|)}$ non-zero entries.
 Since the condition number $\kappa(\Sb^\delta)$ is uniformly bounded due to Lemma~\ref{lem:riesz}, the number of iterations grows only logarithmically in the desired accuracy $\delta$. 
Thus, the cost of $k\simeq 1+|\log(\delta)|$ iterations of the CG method to reach the accuracy can be bounded roughly by $ (1+|\log(\delta)|))^2$.
\begin{lemma}\label{lem:invblock}
For $\delta>0$ as in Lemma~\ref{lem:riesz}, there exists a matrix $\Rb^\delta(\omega)$ such that $\norm{\Sb(\omega)^{-1}-\Rb^\delta(\omega)}{2}\leq \delta$.
Moreover, $\Rb^\delta(\omega)$ satisfies
	\begin{align}\label{eq:invbanded}
d(i,j)>C_{\rm inv} \zeta(|\log(\delta)|^2+1) \text{ or }\level(i)\neq \level(j)\quad\implies\quad \Rb_{ij}^\delta(\omega)=0,
\end{align}
for $\zeta$ from Lemma~\ref{lem:stiffness} and 
$C_{\rm inv}>0$ depending only on $D$, the shape regularity of $\TT_0$ and the contrast 
$\gamma_{\operatorname{max}}/\gamma_{\operatorname{min}}$.
The number of non zero entries is bounded by ${\rm nnz}(\Rb^\delta)\lesssim N(1+|\log(\delta)|)^d$.
\end{lemma}
\begin{proof}
	Due to Lemma~\ref{lem:stiffness} and the fact that $\BBb(\omega)$ is a Riesz basis (Lemma~\ref{lem:riesz}), we observe that all eigenvalues of $\Sb^\delta(\omega)$ are of order $\mathcal{O}(1)$ as
	long as $\delta\lesssim 1$. 
	Therefore, we can obtain $\Rb^\delta(\omega)$ by application of CG steps to $\Sb^{\tilde\delta}(\omega)$ (we chose ${\tilde\delta}>0$ later, see, e.g.,~\cite[Chapter~6]{SaadIter}).
	The convergence
	properties of CG show
	\begin{align*}
	\norm{\Sb^{\tilde\delta}(\omega)^{-1}-\Rb^\delta(\omega)}{2}\leq \delta
	\end{align*}
	if we perform $k=\mathcal{O}(|\log(\delta)|+1)$ CG-steps. This follows since
	\begin{align*}
	 \norm{{\rm res}_k}{\ell_2}\simeq \sqrt{\Sb^{\tilde\delta}(\omega){\rm res}_k\cdot {\rm res}_k}
	 \end{align*}
	for the residual ${\rm res}_k$ of the CG method.
	From Lemma~\ref{lem:stiffness}, we see that $\Rb^\delta(\omega)$ satisfies
	\begin{align*}
	d(i,j)> \zeta(|\log(\delta)|+1)^2 \text{ or }\level(i)\neq \level(j)\quad\implies\quad \Rb_{ij}^\delta(\omega)=0,
	\end{align*}
	since each CG-step increases the bandwidth by the original bandwidth.
	With Lemma~\ref{lem:stiffness}, we conclude the proof by choosing $k\simeq 1+|\log(\delta)|$ and ${\tilde\delta}\simeq \delta$.
\end{proof}

\begin{lemma}\label{lem:avgmatrix}
We define a discrete approximation to $\AA^{-1}$ by
\begin{align*}
R:=\E\bigl[ \big(\Tb^{-T}(\omega)\Sb(\omega)\Tb^{-1}(\omega)\big)^{-1}\bigr]
=\E \bigl[\Tb(\omega)\Sb(\omega)^{-1}\Tb(\omega)^T\bigr].
\end{align*}
For $\delta>0$ as in Lemma~\ref{lem:riesz},
we define a perturbed and truncated version of $R$ by $R^\delta\in\R^{N\times N}$
\begin{align}
(R^\delta)_{(\ell,k)}:=\begin{cases} \Big(\E\bigl[ \Tb^\delta(\omega)\Rb^\delta(\omega)\Tb^\delta(\omega)^T\bigr]\Big)_{(\ell,k)} & \ell + k \leq |\log(\delta)|,\\
                        0 &\text{else.}
                       \end{cases}
\end{align}
which satisfies $\norm{R-R^\delta}{2}\leq CL^2\delta$. 
The number of non-zero entries in $R^\delta$ is bounded by $\operatorname{nnz}(R^\delta)\lesssim L/\delta^d$.
The constant $C>0$ depends only on $D$, the shape regularity of $\TT_0$ and the contrast 
$\gamma_{\operatorname{max}}/\gamma_{\operatorname{min}}$.
\end{lemma}
\begin{proof}
We define the auxiliary operator 
\begin{align*}
 \widetilde R^\delta:=\E\bigl[ \Tb^\delta(\omega)\Rb^\delta(\omega)\Tb^\delta(\omega)^T\bigr].
\end{align*}
Analogously to matrix sub-blocks, we may partition vectors $x\in \R^N$ by $x = (x_{(1)},\ldots,x_{(L)})$ with $x_{(\ell)}\in \R^{\#\HH_\ell}$. Using this notation and following the proofs of Lemma~\ref{lem:stiffness}, Lemma~\ref{lem:basistrans}, and Lemma~\ref{lem:invblock}, we show
\begin{align*}
 \norm{(R-\widetilde R^\delta)x}{\ell_2}^2\lesssim\sum_{\ell=1}^L\biggl\|\sum_{k=1}^L  (R_{(\ell,k)}-\widetilde R_{(\ell,k)}^\delta)x_{(k)}\biggr\|_{\ell_2}^2\lesssim
 \delta^2 L^4 \sum_{\ell=1}^L\sum_{k=1}^L  \norm{x_{(k)}}{\ell_2}^2=\delta^2 L^4\norm{x}{\ell_2}^2
\end{align*}
and hence $\norm{R-\widetilde R^\delta}{2}\lesssim \delta L^2$.
The estimate~\eqref{eq:Tdecay} implies for $\ell+k>|\log(\delta)|$ 
\begin{align*}
\norm{(\widetilde R^\delta- R^\delta)_{(\ell,k)}}{2}&\leq 
\sum_{j=0}^L \norm{\Tb^\delta(\omega)_{(\ell,j)}}{2}\norm{\Rb^\delta(\omega)_{(j,j)}}{2}\norm{(\Tb^\delta(\omega)_{(k,j)})}{2}\\
&\lesssim 
\sum_{j=0}^L h_\ell(1+\delta)h_k\\
 &\lesssim L2^{-\ell-k}.
\end{align*}
This implies for $x\in\R^N$
\begin{align*}
 \norm{(\widetilde R^\delta- R^\delta)x}{\ell_2}^2&\leq \sum_{i,j=0}^L   \norm{(\widetilde R^\delta- R^\delta)_{(i,j)}x|_{(j)}}{\ell_2}^2\lesssim \sum_{j=0}^L \norm{x|_{(j)}}{\ell_2}^2\sum_{i=|\log(\delta)|-j}^{L} L 2^{-i-j}\\
 &\lesssim L\delta \norm{x}{\ell_2}^2.
\end{align*}
The number of non-zero entries in $R^\delta$ can be bounded sufficiently by ignoring the sparsity within the blocks and just summing up the entries
\begin{align*}
 \sum_{0\leq i+j\leq |\log(\delta)|} \# (R^\delta)_{(i,j)}\lesssim \sum_{0\leq i+j\leq |\log(\delta)|} 2^{d(i+j)}\lesssim L \delta^{-d},
\end{align*}
where we used that $(R^\delta)_{i,j}\in \R^{\#\HH_i\times \#\HH_j}$ and $\#\HH_i\simeq 2^{di}$.
This concludes the proof.
\end{proof}

To formulate the following main theorem, we identify the matrix $R^\delta$ with an operator $\RR^\delta\colon L^2(D)\to L^2(D)$
via the natural embedding $\iota \colon \R^N \to {\rm span}(\HH)$, $\iota(\alpha)=\sum_{i=1}^N \alpha_i\phi_i\in L^2(D)$.
There holds $\RR^\delta := \iota R^\delta \iota^\star$.

\begin{theorem}\label{t:main}
For a given accuracy $\delta>0$ with $\delta(1+|\log(\delta)|) \lesssim 1$ sufficiently small, there exists a finite dimensional operator $\RR^\delta\colon L^2(D)\to L^2(D)$ which depends only on $\delta$ such that
\begin{align*}
\norm{\AA^{-1}-\RR^\delta}{\mathcal{L}(L^2(D),L^2(D))}\leq \delta.
\end{align*}
The corresponding operator matrix $R^\delta$ from Lemma~\ref{lem:avgmatrix} has at most 
$\mathcal{O}(|\log(\delta)|^{2d+1}\delta^{-d})$ non-zero entries.
The hidden constant depends only on $D$, the shape regularity of $\TT_0$ and the contrast 
$\gamma_{\operatorname{max}}/\gamma_{\operatorname{min}}$.
\end{theorem}
\begin{proof}[Constructive proof]
	We use the operator matrix $R^\delta\in\R^{N\times N}$ from Lemma~\ref{lem:avgmatrix}.
	Given $f\in L^2(D)$, define $\Fb(\omega)\in\R^N$ by $\Fb_i(\omega):=(f,\bb_i/\enorm{\bb_i}_\omega)$.
	By definition, there holds $\Sb(\omega)\alphab(\omega) = \Fb(\omega)$ with $\ub_L(\omega):= \sum_{i=1}^N \alphab_i(\omega) \bb_i/\enorm{\bb_i}_\omega\in {\rm span}(\BBb(\omega))$ being 
	the Galerkin approximation to $\ub(\omega)\in H^1_0(\Omega)$.
	Galerkin orthogonality
		$$\ab_\omega(\ub(\omega)-\ub_L(\omega),{\rm span}(\BBb(\omega)))=0$$
		implies $\ub(\omega)-\ub_L(\omega)\in W_L$ and, hence, $\Pi_L(\ub(\omega)-\ub_L(\omega))=0$. Thus, 
		\begin{align*}
		\norm{\ub(\omega)-\Pi_L\ub_L(\omega)}{L^2(D)}\leq \norm{(1-\Pi_L)\ub(\omega)}{L^2(D)}\lesssim 
		h_L\norm{f}{L^2(D)}
		\end{align*}
		using standard approximation properties of piecewise constants (Poincar\'e inequality) and a standard energy bound.
	With the transfer matrices $\Tb(\omega)$ from Lemma~\ref{lem:basistrans}, we obtain
	\begin{align*}
	\widetilde F:=\iota^\star f = \Tb^{-T}(\omega) \Fb(\omega)
	\end{align*}
	and hence $\betab\in\R^N$ with $\Tb^{-T}(\omega)\Sb(\omega)\Tb^{-1}(\omega)\betab(\omega) =\widetilde F$ satisfies $\Tb(\omega)\alphab(\omega) =\betab(\omega)$. Together with Lemma~\ref{lem:basistrans}, this shows that $\Pi_L\ub_L(\omega) = \sum_{i=1}^N \betab_i(\omega) \phi_i/\norm{\phi_i}{L^2(D)}$.
	The approximate solution $\RR^\delta f= \sum_{i=1}^N \gamma_i \phi_i/\norm{\phi_i}{L^2(D)}$ with $\gamma:=R^\delta \widetilde F$ satisfies 
	\begin{align*}
	|\gamma-\E[\betab]|\lesssim L^2\delta\norm{f}{L^2(D)},
	\end{align*}
	by use of Lemma~\ref{lem:avgmatrix} and since $\E_M[\betab]=R\widetilde F$. Since $\HH$ is an orthogonal basis, we obtain immediately 
	$\norm{\RR^\delta f - \RR f}{L^2(D)}\lesssim L^2\delta\norm{f}{L^2(D)}$, where
	\begin{align*}
	\RR^\delta f=\E[\bu_L].
	\end{align*}
	Combining the above error bounds, we conclude
	\begin{align*}
	\norm{\E[\ub]-\RR^\delta f}{L^2(D)}\lesssim (L^2\delta+h_L)\norm{f}{L^2(D)}.
	\end{align*}
	With $L\simeq|\log\delta|$ and $h_L\simeq \delta$ there holds $\norm{\E[\ub]-\RR^\delta f}{L^2(D)}\lesssim (1+|\log(\delta)|^2)\delta\norm{f}{L^2(D)}$. Replacing $\delta$ with $\delta/L^2$, we conclude the proof. 
\end{proof}

\section{Sparse operator compression}\label{s:sos}
Theorem~\ref{t:main} shows that the expected operator can indeed be compressed to a sparse matrix. The constructive proof motivates a compression algorithm by simply replacing the expectation by a suitable sample mean. For this purpose, let  $\Omega_M\subset \Omega$ be a finite set of sampling points with $|\Omega_M|=M\in\N$ and define the sample mean $\E_M[\mathbf{X}]:=M^{-1}\sum_{\omega\in\Omega_M}\mathbf{X}(\omega)$ for a random field $\mathbf{X}$. 
It is readily seen that Lemma~\ref{lem:avgmatrix} remains valid when $\E$ is replaced by $\E_M$. More precisely, define 
	\begin{align*}
	R_M:=\E_M\bigl[\big(\Tb^{-T}(\omega)\Sb(\omega)\Tb^{-1}(\omega)\big)^{-1}\bigr]
	=\E_M\bigl[\Tb(\omega)\Sb(\omega)^{-1}\Tb(\omega)^T\bigr]
	\end{align*}
	and a perturbed and truncated version of $R_M$ by $R_M^\delta\in\R^{N\times N}$
	\begin{align}\label{e:trunc}
	(R_M^\delta)_{(\ell,k)}:=\begin{cases} \Big(\E_M\bigl[ \Tb^\delta(\omega)\Rb^\delta(\omega)\Tb^\delta(\omega)^T\bigr]\Big)_{(\ell,k)} & \ell + k \leq |\log(\delta)|,\\
	0 &\text{else.}
	\end{cases}
	\end{align}
Then 
\begin{equation}
\norm{R_M-R_M^\delta}{2}\leq CL^2\delta
\end{equation}
and the number of non-zero entries in $R_M^\delta$ is bounded by $\mathcal{O}(L/\delta^d)$.

\begin{remark}\label{r:trunc}
	The truncation condition $\ell + k \leq |\log(\delta)|$ in \eqref{e:trunc} can be relaxed to $\ell + k \leq C|\log(\delta)|$ for some $C\simeq 1$ without any harm.
	In practice, when $L\simeq|\log\delta|$ is chosen, a natural choice would be $\ell + k \leq L$. In the numerical experiment of Section~\ref{s:numexp} 
	we will see that sometimes it can be advantageous to include a few more blocks of the lower right part of the matrix (see Eq.~\eqref{e:tildeR}) to recover gradient information. 
\end{remark}
The analog of Theorem~\ref{t:main} in this discrete stochastic setting then reads.
\begin{corollary}
For given an accuracy $\delta>0$ as in Theorem~\ref{t:main} and a set of $M$ samples $\Omega_M\subset \Omega$, $M\in\N$, there exists a finite dimensional operator $\RR_M^\delta\colon L^2(D)\to L^2(D)$ which depends only on the sample coefficients $\Ab(\omega)$, $\omega\in\Omega_M$, $\delta$, and $D$, such that
	\begin{align*}
	\norm{\AA^{-1}-\RR_M^\delta}{\mathcal{L}(L^2(D),L^2(D))}\leq \delta + 
	\norm{(\E-\E_M)[\AAb^{-1}]}{\mathcal{L}(L^2(D),L^2(D))}.
	\end{align*}
	The corresponding operator matrix $R_M^\delta$ has $\mathcal{O}(|\log(\delta)|^{2d+1}\delta^{-d})$ non-zero entries.
The hidden constant depends only on $D$, the shape regularity of $\TT_0$ and the contrast 
$\gamma_{\operatorname{max}}/\gamma_{\operatorname{min}}$.
\end{corollary}
When using a plain Monte Carlo sampling the mean squared sampling error scales like $M^{-1}$
meaning that $M\simeq\delta^{-2}$ samples suffice to ensure that the sampling error is not dominating the error bound. 
This is optimal in the present setting with no assumptions on the distribution of the random diffusion coefficient.
More advanced sampling techniques such as quasi Monte Carlo methods are certainly possible under additional assumptions such as a rapid decay of eigenvalues of a given Karhunen-Lo\`eve expansion
of the random parameter (see~\cite{hoqmc} for a discussion in terms of PDEs with random parameters). 
Even more promising is the possible intertwining of the hierarchical decomposition and the sampling procedure in the spirit of multilevel/multi-index Monte Carlo (see, e.g.,~\cite{giles,mimi1} for the seminal works
as well as~\cite{mifem}). At least in the regime where stochastic homogenization applies, the computation of basis functions is likely to be essentially independent of the parameter $\omega$ for
levels that are much coarser than the characteristic length scale of random oscillation (or correlation) \cite{GallistlPeterseim2017random}. This has been made rigorous in a two-level setting in \cite{2019arXiv191211646F}. The increasing variance for the levels approaching the scale of correlation, stationarity could be exploited to improve the overall complexity.
\smallskip

Another interesting case is the use of log-normal coefficients $\Ab(\omega) = \exp(\boldsymbol{Z}(\omega))$ for a normal random field $\boldsymbol{Z}$.
As shown in~\cite{hgauss}, such random fields can be efficiently generated for general covariance functions and non-uniform grids. The present analysis, however, breaks
down since the assumption of bounded contrast in~\eqref{e:classM} is violated. The authors are confident, however, that the arguments can be modified in the sense that
the extreme contrast samples will only appear with very low probability (the tails of the Gaussian density). Thus, a polynomial dependence on the contrast (as is observed for the present construction)
will not perturb the final result.
\smallskip

We shall finally mention that so far the construction relies on the exact solution of the (infinite-dimensional) corrector problems \eqref{eq:psi} and their preconditioned variant, respectively. The elegant way to transfer all results to a fully discrete setting is to consider a space-discrete problem from the very beginning. It is readily seen that all constructions and results remain valid if we replace the space $V=H^1_0(D)$ by a suitable finite dimensional subspace $V_h\subset V$ throughout the paper. We have in mind some standard $V$-conforming finite element space $V_h$ that is based on some regular mesh of width $h$ which turns the preconditioned corrector problems into finite element problems on the mesh $h$ restricted to local subdomains of diameter $h_\ell|\log\delta|$. The only restriction that comes with this discretization step is that the mesh size $h$ limits the number of possible levels $L$ in the hierarchical decomposition and, hence, the possible accuracy $\delta\lesssim h$ when the sparse approximation is compared with the reference solution $\E[\ub_h]$ where $\ub_h$ solves \eqref{e:modelweak} with $V$ replaced with $V_h$. Clearly, the overall accuracy of the fully discrete method depends on the error $\|\E[\ub-\ub_h]\|_{L^2(D)}$ which is a standard finite element error that depends on the spatial regularity of $\Ab$ and also its possible frequencies of oscillations. All this is well understood and implies the usual conditions on the smallness of $h$ so that $\Ab$ is properly resolved (see e.g. \cite{Peterseim2012}). 

\section{Numerical experiment}\label{s:numexp}
This section presents some simple numerical experiments to illustrate 
the performance of the method. We consider the domain $D=[0,1]^d$ for 
$d=1,2$ and the coefficient $\Ab$ is scalar i.i.d.\ and, on each cell 
of the uniform Cartesian mesh $\TT_{\varepsilon}$, it is uniformly 
distributed in the interval 
$[\gamma_{\operatorname{min}},\gamma_{\operatorname{max}}]=[0.5,10]$. 
The mesh width (scale of oscillation/correlation length) is $\varepsilon=2^{-8}$ ($d=1$) and $\varepsilon=2^{-5}$ ($d=2$).

The approximations of the solution operator are based on sequences of 
uniform Cartesian meshes $\TT_\ell$ ($\ell=0,1,2,\ldots,L$) of mesh 
width $h_\ell=2^{-\ell}$ that do not necessarily resolve $\varepsilon$. 
We compute approximations $\RR^L=\RR^\delta_{M_L}$ of the expected solution 
operator depending on the maximal level $L$ which means that we expect $L^2(D)$ 
errors of order $\delta\approx 2^{-L}$. The truncation of blocks is performed based on the criterion $k+\ell\leq L$ as indicated in Remark~\ref{r:trunc}. For the solution of the 
corrector problems and the reference solution $\mathbf{u}_h$ we use 
$d$-linear finite elements on the mesh $\TT_{h}$ where $h=2^{-14}$ 
($d=1$) and $h=2^{-9}$ ($d=2$). To achieve accuracy 
of order $\delta$ (w.r.t. to the reference solution) we perform $\lceil 
L/2\rceil$ CG-iterations for both computing the correctors $\CCb^\delta(\omega)$ 
and inverting the block-diagonal stiffness matrices 
$\Sb^\delta(\omega)$. For the approximation of the expected values we 
use a quasi-Monte Carlo method (particularly a Sobol sequence) with appropriate numbers of sampling 
points $M_h:=h^{-1}$ for the reference solutions and $M_L:=2^{L}$ for 
the approximations. While we did not show that the problem is smooth enough to justify the use of quasi-Monte Carlo sampling, we still observe the expected higher
convergence rate compared to plain Monte Carlo sampling and thus save significant compute time.
\smallskip

Since the computation of a reference expected operator is hardly 
feasible we only compute the error for one non-smooth deterministic 
right-hand side $f=\chi_{[.5,1]\times[0,1]^{d-1}}\in L^2(D)\setminus H^1(D)$.
Figures~\ref{fig:numexp1d}--\ref{fig:numexp2d} (left plots) depict the 
errors $\|\E_{M_h}[\ub_h]-\RR^L f\|_{L^2(D)}$ versus the number of nonzero 
entries of $\operatorname{nnz}(R^L)$ for $L=1,2,\ldots\;$. The results are 
very well in agreement (up to a, possibly pessimistic, logarithmic factor) with the prediction that 
\begin{align*}
\|\E_{M_h}[\ub_h]-\RR^L f\|_{L^2(D)}\lesssim 
M_L^{-1}+\frac{|\log(\operatorname{nnz}(R^L))^{2+1/d}}{\operatorname{nnz}(R^L)^{1/d}}
\end{align*}
for $d=1,2$. 
This is the optimal rate of convergence (up to a logarithmic 
factor) given a piecewise constant approximation.
\smallskip

In this setting where the expected solution $\E[\ub]$ is even $H^2(D)$ 
regular it would be desirable to recover gradient information from the 
piecewise constant approximation by suitable postprocessing, e.g., in the hierarchical basis associated with a constant coefficient. Figures~\ref{fig:numexp1d}--\ref{fig:numexp2d} (left plots) indicate that this is not automatically achieved for non-smooth right-hand sides with the present choice of parameters. However, when the 
truncation in \eqref{e:trunc} is slightly relaxed in the following form
\begin{align}\label{e:tildeR}
(\tilde R^L)_{(\ell,k)}:=\begin{cases} \Big(\E_M\bigl[ 
\Tb^\delta(\omega)\Rb^\delta(\omega)\Tb^\delta(\omega)^T\bigr]\Big)_{(\ell,k)} 
& \ell + k \leq L+\max(1,\lceil\log_2 L\rceil),\\
0 &\text{else,}
\end{cases}
\end{align}
accurate reconstruction of gradients seems possible. From this slightly more accurate but slightly more 
dense approximation $\tilde R^L$ we can reconstruct the coefficients of a smooth 
approximation $u^{1}_L\in\operatorname{span}\BBb(\omega_\Delta)$ (with $\omega_\Delta\in\Omega$ such that $\Ab(\omega_\Delta)=1$) in the hierarchical basis that corresponds to the Laplacian by simply 
applying $T^\delta(\omega_\Delta)^{-1}$ to $\tilde R^Lf$. The errors of this smooth 
postprocessing $\|\nabla (\E_{M_h}[\ub_h]-u^{1}_L)\|_{L^2(D)}$ are 
plotted in Figures~\ref{fig:numexp1d}--\ref{fig:numexp2d} (right plots)
against the number of non-zero entries 
$\operatorname{nnz}(\tilde{R}^L)$. The observed rate of convergence for 
the $H^1$-error is $\operatorname{nnz}(\tilde{R}^L)^{-1/d}$ (up to a 
logarithmic factor) which is nearly optimal. See also the plots on the left of Figures~\ref{fig:numexp1d}--\ref{fig:numexp2d} which indicate that the step from \eqref{e:trunc} to \eqref{e:tildeR} is essential for meaningful gradient reconstruction. 
\begin{figure}
     \centering
\includegraphics[width=0.48\linewidth]{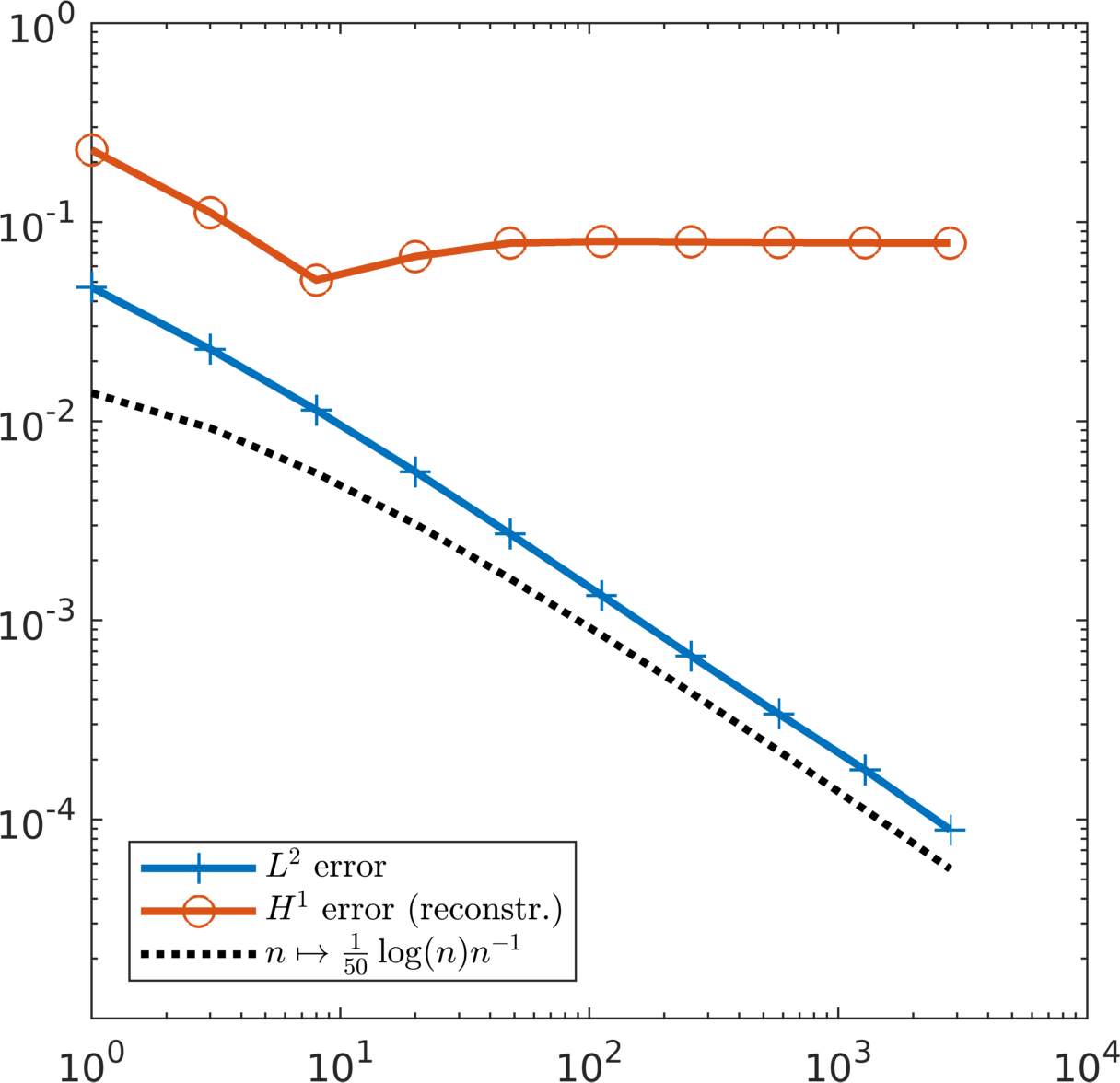}\hspace{0.01\linewidth}
     \includegraphics[width=0.48\linewidth]{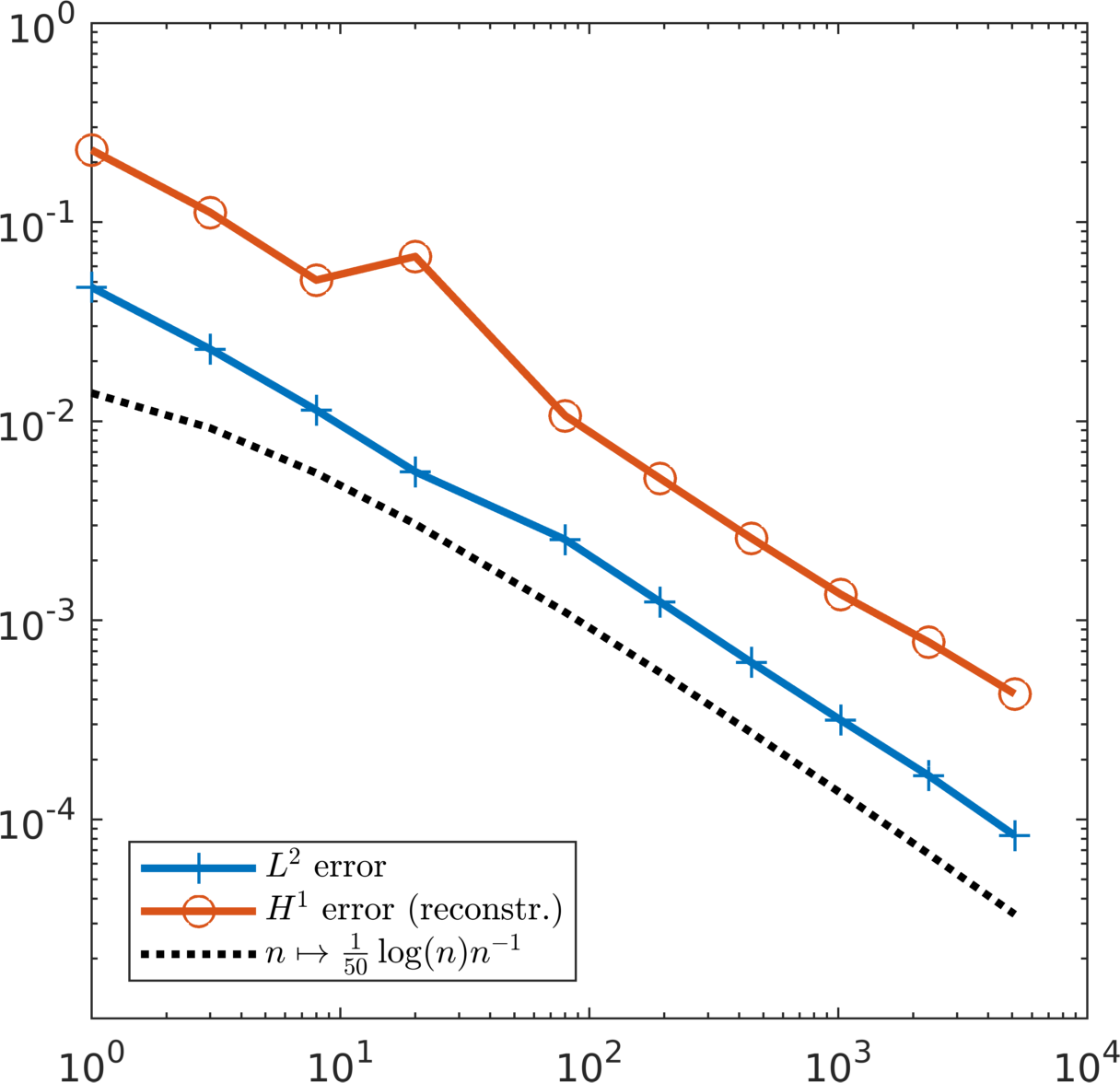}
     \caption{Numerical results in $1d$: $L^2(D)$-errors 
$\|\E_{M_h}[\ub_h]-\RR^L f\|_{L^2(D)}$ and $H^1(D)$-errors $\|\nabla 
(\E_{M_h}[\ub_h]-u^{1}_L)\|_{L^2(D)}$ of post-processed approximation 
for $L=1,2,\ldots,10$. Left: Errors versus $\operatorname{nnz}(R^{L})$ 
using original approach \eqref{e:trunc}. Right: Errors versus 
$\operatorname{nnz}(\tilde R^{L})$ using modified approach 
\eqref{e:tildeR}.}
     \label{fig:numexp1d}
\end{figure}

\begin{figure}
     \centering
\includegraphics[width=0.48\linewidth]{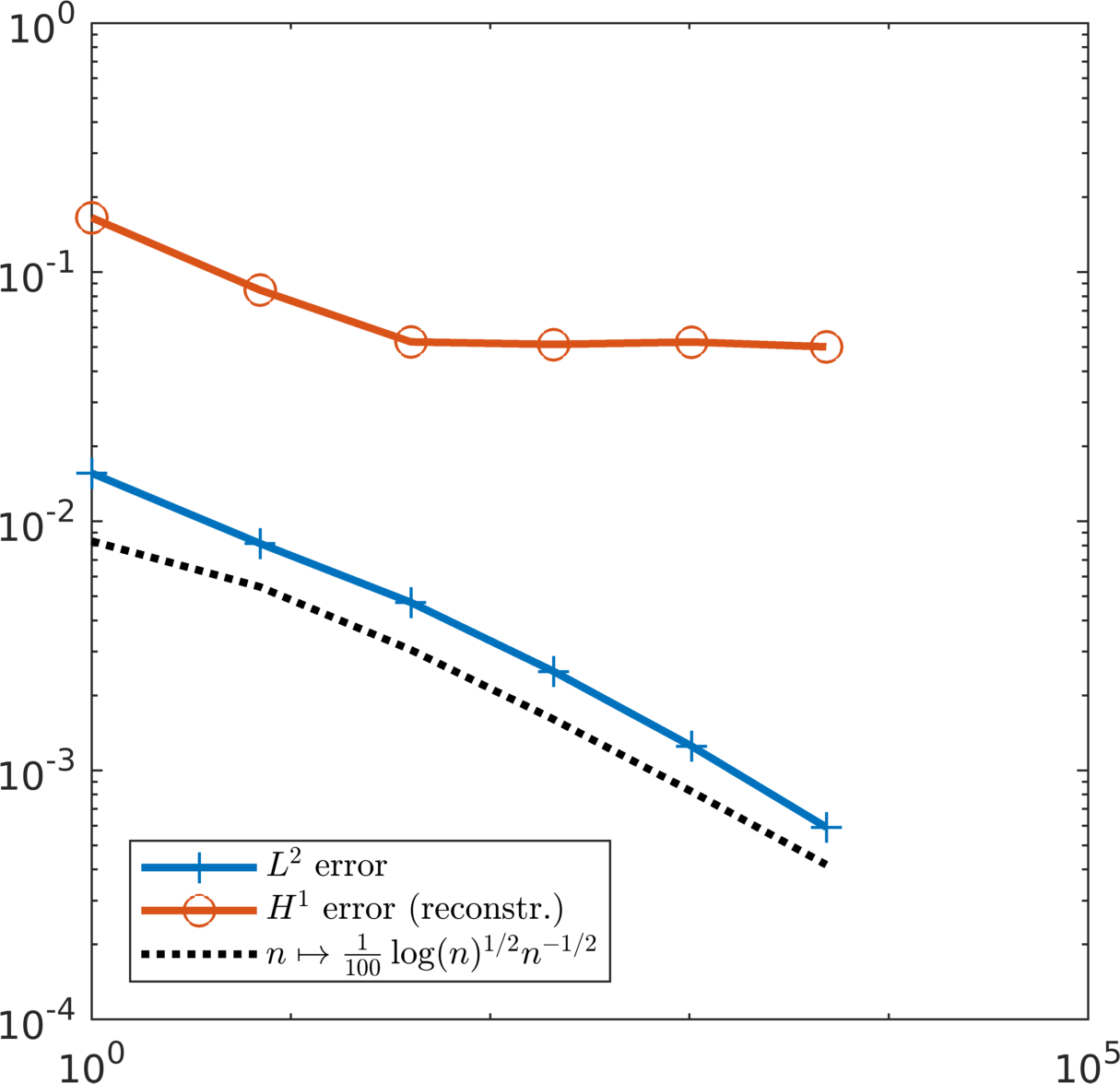}\hspace{0.01\linewidth}
     \includegraphics[width=0.48\linewidth]{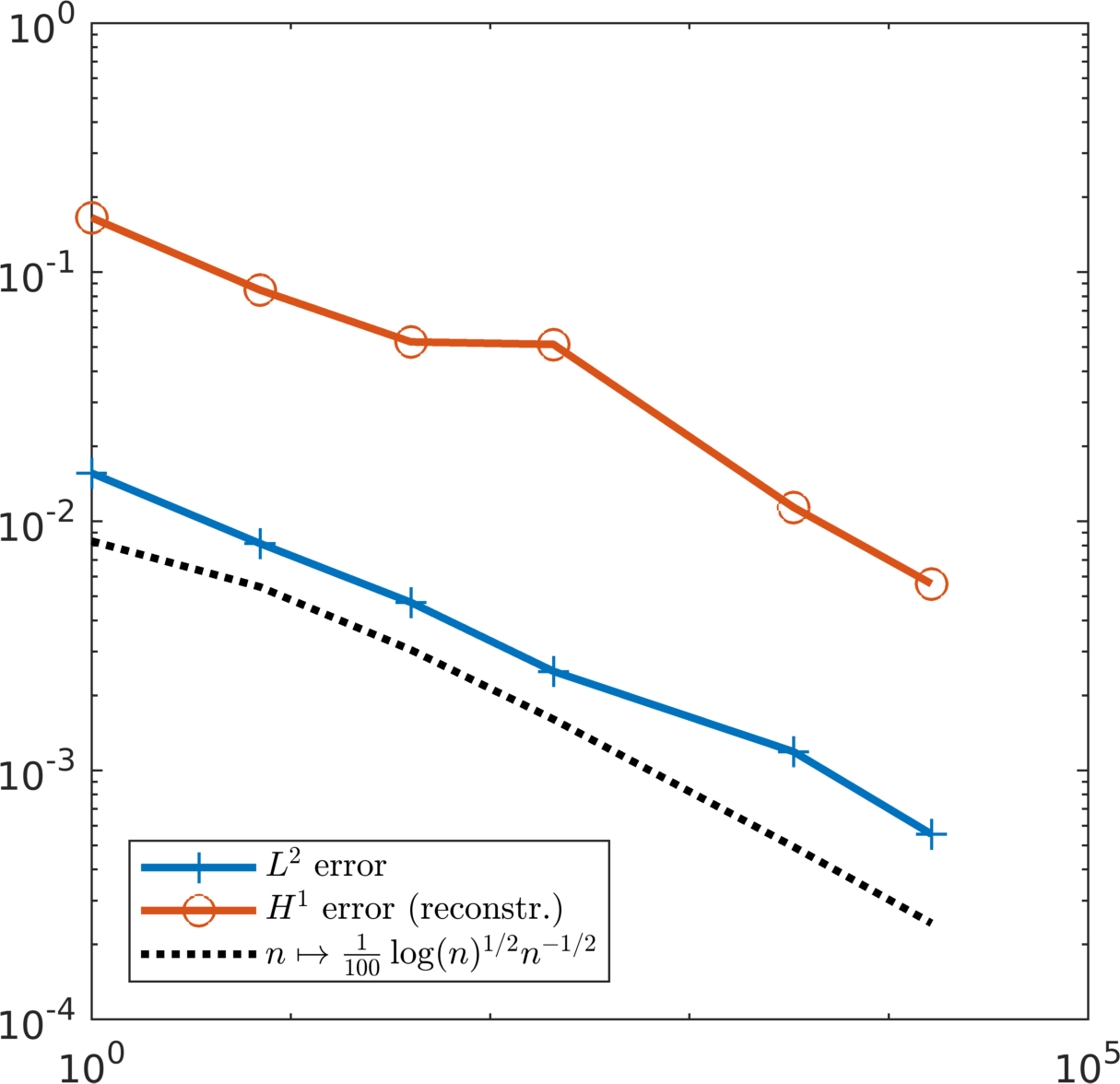}
     \caption{Numerical results in $2d$: $L^2(D)$-errors 
$\|\E_{M_h}[\ub_h]-\RR^L f\|_{L^2(D)}$ and $H^1(D)$-errors $\|\nabla 
(\E_{M_h}[\ub_h]-u^{1}_L)\|_{L^2(D)}$ of post-processed approximation  
for $L=1,2,\ldots,6$. Left: Errors versus $\operatorname{nnz}(R^{L})$ 
using original approach \eqref{e:trunc}. Right: Errors versus 
$\operatorname{nnz}(\tilde R^{L})$ using modified approach 
\eqref{e:tildeR}.}
     \label{fig:numexp2d}
\end{figure}

These first numerical results support the theoretical findings and 
indicate the potential of the approach. Since the techniques that were 
used in the construction of the method and its analysis, in particular 
the localized orthogonal decomposition, generalize in a straight-forward 
way to other classes of operators such as linear elasticity 
\cite{MR3548556} or Helmholtz problems 
\cite{MR3614010,GallistlPeterseim2015,MR3718330}, we believe that the 
sparse compression algorithm for the approximation of expected solution 
operators is applicable beyond the prototypical model problem of this paper.

\bibliographystyle{plain}
\bibliography{literature}

\begin{thebibliography}{10}

\bibitem{ArmstrongKuusiMourrat2017}
S.~Armstrong, T.~Kuusi, and J.-C. Mourrat.
\newblock The additive structure of elliptic homogenization.
\newblock {\em Inventiones mathematicae}, 208(3):999--1154, Jun 2017.

\bibitem{coh2}
M.~Bachmayr, A.~Cohen, R.~DeVore, and G.~Migliorati.
\newblock Sparse polynomial approximation of parametric elliptic {PDE}s. {P}art
  {II}: {L}ognormal coefficients.
\newblock {\em ESAIM Math. Model. Numer. Anal.}, 51(1):341--363, 2017.

\bibitem{coh1}
M.~Bachmayr, A.~Cohen, and G.~Migliorati.
\newblock Sparse polynomial approximation of parametric elliptic {PDE}s. {P}art
  {I}: {A}ffine coefficients.
\newblock {\em ESAIM Math. Model. Numer. Anal.}, 51(1):321--339, 2017.

\bibitem{Bourgain2018}
J.~Bourgain.
\newblock On a homogenization problem.
\newblock {\em Journal of Statistical Physics}, 2018.

\bibitem{BourgeatPiatnitski2004}
A.~Bourgeat and A.~Piatnitski.
\newblock Approximations of effective coefficients in stochastic
  homogenization.
\newblock {\em Ann. Inst. H. Poincar\'e Probab. Statist.}, 40(2):153--165,
  2004.

\bibitem{MR3718330}
D.~L. Brown, D.~Gallistl, and D.~Peterseim.
\newblock Multiscale {P}etrov-{G}alerkin method for high-frequency
  heterogeneous {H}elmholtz equations.
\newblock In {\em Meshfree methods for partial differential equations {VIII}},
  volume 115 of {\em Lect. Notes Comput. Sci. Eng.}, pages 85--115. Springer,
  Cham, 2017.

\bibitem{homl}
J.~Dick, R.~N. Gantner, Q.~T. Le~Gia, and C.~Schwab.
\newblock Multilevel higher-order quasi-{M}onte {C}arlo {B}ayesian estimation.
\newblock {\em Math. Models Methods Appl. Sci.}, 27(5):953--995, 2017.

\bibitem{hoqmc}
J.~Dick, F.~Y. Kuo, Q.~T. Le~Gia, D.~Nuyens, and C.~Schwab.
\newblock Higher order {QMC} {P}etrov-{G}alerkin discretization for affine
  parametric operator equations with random field inputs.
\newblock {\em SIAM J. Numer. Anal.}, 52(6):2676--2702, 2014.

\bibitem{dick}
Josef Dick.
\newblock Walsh spaces containing smooth functions and quasi-{M}onte {C}arlo
  rules of arbitrary high order.
\newblock {\em SIAM J. Numer. Anal.}, 46(3):1519--1553, 2008.

\bibitem{mifem}
Josef Dick, Michael Feischl, and Christoph Schwab.
\newblock Improved efficiency of a multi-index {FEM} for computational
  uncertainty quantification.
\newblock {\em SIAM J. Numer. Anal.}, 57(4):1744--1769, 2019.

\bibitem{DuerinckxGloriaOtto2016}
M.~Duerinckx, A.~Gloria, and F.~Otto.
\newblock The structure of fluctuations in stochastic homogenization.
\newblock {\em arXiv e-prints}, 1602.01717 [math.AP], 2016.

\bibitem{Duer19}
Mitia Duerinckx, Antoine Gloria, and Marius Lemm.
\newblock A remark on a surprising result by bourgain in homogenization.
\newblock {\em Communications in Partial Differential Equations},
  44(12):1345--1357, 2019.

\bibitem{doi:10.1137/120863162}
D.~Elfverson, E.~Georgoulis, and A.~M{\aa}lqvist.
\newblock An adaptive discontinuous galerkin multiscale method for elliptic
  problems.
\newblock {\em Multiscale Modeling \& Simulation}, 11(3):747--765, 2013.

\bibitem{Elfverson.Georgoulis.Mlqvist.ea:2012}
D.~{Elfverson}, E.~H. {Georgoulis}, A.~{M{\aa}lqvist}, and D.~{Peterseim}.
\newblock Convergence of a discontinuous galerkin multiscale method.
\newblock {\em SIAM J. Numer. Anal.}, 51(6):3351--3372, 2013.

\bibitem{hgauss}
Michael Feischl, Frances~Y. Kuo, and Ian~H. Sloan.
\newblock Fast random field generation with {$H$}-matrices.
\newblock {\em Numer. Math.}, 140(3):639--676, 2018.

\bibitem{2019arXiv191211646F}
J.~{Fischer}, D.~{Gallistl}, and D.~{Peterseim}.
\newblock {A priori error analysis of a numerical stochastic homogenization
  method}.
\newblock {\em arXiv e-prints}, page arXiv:1912.11646, Dec 2019.

\bibitem{GallistlPeterseim2015}
D.~{Gallistl} and D.~{Peterseim}.
\newblock Stable multiscale {P}etrov-{G}alerkin finite element method for high
  frequency acoustic scattering.
\newblock {\em Comput. Methods Appl. Mech. Eng.}, 295:1--17, 2015.

\bibitem{GallistlPeterseim2017}
D.~Gallistl and D.~Peterseim.
\newblock Computation of quasi-local effective diffusion tensors and
  connections to the mathematical theory of homogenization.
\newblock {\em Multiscale Model. Simul.}, 15(4):1530--1552, 2017.

\bibitem{GallistlPeterseim2017random}
D.~{Gallistl} and D.~{Peterseim}.
\newblock {Numerical stochastic homogenization by quasilocal effective
  diffusion tensors}.
\newblock {\em ArXiv e-prints}, 1702.08858, 2017.

\bibitem{giles}
M.~B. Giles.
\newblock Multilevel {M}onte {C}arlo path simulation.
\newblock {\em Oper. Res.}, 56(3):607--617, 2008.

\bibitem{GloriaOtto2015}
A.~Gloria and F.~Otto.
\newblock The corrector in stochastic homogenization: optimal rates, stochastic
  integrability, and fluctuations.
\newblock {\em arXiv e-prints}, 1510.08290 [math.AP], 2015.

\bibitem{GloriaOtto2017}
A.~Gloria and F.~Otto.
\newblock Quantitative results on the corrector equation in stochastic
  homogenization.
\newblock {\em J. Eur. Math. Soc. (JEMS)}, 19(11):3489--3548, 2017.

\bibitem{mimi1}
A.-L. Haji-Ali, F.~Nobile, and R.~Tempone.
\newblock Multi-index {M}onte {C}arlo: when sparsity meets sampling.
\newblock {\em Numer. Math.}, 132(4):767--806, 2016.

\bibitem{MR3548556}
P.~Henning and A.~Persson.
\newblock A multiscale method for linear elasticity reducing {P}oisson locking.
\newblock {\em Comput. Methods Appl. Mech. Engrg.}, 310:156--171, 2016.

\bibitem{HenningPeterseim2013}
P.~Henning and D.~Peterseim.
\newblock Oversampling for the multiscale finite element method.
\newblock {\em Multiscale Model. Simul.}, 11(4):1149--1175, 2013.

\bibitem{MR3783084}
T.~Y. Hou, D.~Huang, K.~C. Lam, and P.~Zhang.
\newblock An {A}daptive {F}ast {S}olver for a {G}eneral {C}lass of {P}ositive
  {D}efinite {M}atrices {V}ia {E}nergy {D}ecomposition.
\newblock {\em Multiscale Model. Simul.}, 16(2):615--678, 2018.

\bibitem{MR3736719}
T.~Y. Hou and P.~Zhang.
\newblock Sparse operator compression of higher-order elliptic operators with
  rough coefficients.
\newblock {\em Res. Math. Sci.}, 4:Paper No. 24, 49, 2017.

\bibitem{2018arXiv180410260K}
Jongchon Kim and Marius Lemm.
\newblock On the averaged {G}reen's function of an elliptic equation with
  random coefficients.
\newblock {\em Arch. Ration. Mech. Anal.}, 234(3):1121--1166, 2019.

\bibitem{KorPY18}
R.~Kornhuber, D.~Peterseim, and H.~Yserentant.
\newblock An analysis of a class of variational multiscale methods based on
  subspace decomposition.
\newblock {\em Math. Comp.}, 87(314):2765--2774, 2018.

\bibitem{KorY16}
R.~Kornhuber and H.~Yserentant.
\newblock Numerical homogenization of elliptic multiscale problems by subspace
  decomposition.
\newblock {\em Multiscale Model. Simul.}, 14(3):1017--1036, 2016.

\bibitem{Kozlov1979}
S.~M. Kozlov.
\newblock The averaging of random operators.
\newblock {\em Mat. Sb. (N.S.)}, 109(151)(2):188--202, 327, 1979.

\bibitem{MalP14}
A.~M{\aa}lqvist and D.~Peterseim.
\newblock Localization of elliptic multiscale problems.
\newblock {\em Math. Comp.}, 83(290):2583--2603, 2014.

\bibitem{gamblets}
H.~Owhadi.
\newblock Multigrid with rough coefficients and multiresolution operator
  decomposition from hierarchical information games.
\newblock {\em SIAM Review}, 59(1):99--149, 2017.

\bibitem{OWHADI201799}
H.~Owhadi and L.~Zhang.
\newblock Gamblets for opening the complexity-bottleneck of implicit schemes
  for hyperbolic and parabolic odes/pdes with rough coefficients.
\newblock {\em Journal of Computational Physics}, 347:99 -- 128, 2017.

\bibitem{MR3971243}
Houman Owhadi and Clint Scovel.
\newblock {\em Operator-adapted wavelets, fast solvers, and numerical
  homogenization}, volume~35 of {\em Cambridge Monographs on Applied and
  Computational Mathematics}.
\newblock Cambridge University Press, Cambridge, 2019.
\newblock From a game theoretic approach to numerical approximation and
  algorithm design.

\bibitem{PapanicolaouVaradhan1981}
G.~C. Papanicolaou and S.~R.~S. Varadhan.
\newblock Boundary value problems with rapidly oscillating random coefficients.
\newblock In {\em Random fields, {V}ol. {I}, {II} ({E}sztergom, 1979)},
  volume~27 of {\em Colloq. Math. Soc. J\'anos Bolyai}, pages 835--873.
  North-Holland, Amsterdam-New York, 1981.

\bibitem{Peterseim2016}
D.~Peterseim.
\newblock Variational multiscale stabilization and the exponential decay of
  fine-scale correctors.
\newblock In Gabriel~R. Barrenechea, Franco Brezzi, Andrea Cangiani, and
  Emmanuil~H. Georgoulis, editors, {\em Building Bridges: Connections and
  Challenges in Modern Approaches to Numerical Partial Differential Equations},
  pages 343--369. Springer International Publishing, Cham, 2016.

\bibitem{MR3614010}
D.~Peterseim.
\newblock Eliminating the pollution effect in {H}elmholtz problems by local
  subscale correction.
\newblock {\em Math. Comp.}, 86(305):1005--1036, 2017.

\bibitem{Peterseim2012}
D.~Peterseim and S.~Sauter.
\newblock Finite elements for elliptic problems with highly varying,
  nonperiodic diffusion matrix.
\newblock {\em Multiscale Modeling \& Simulation}, 10(3):665--695, 2012.

\bibitem{SaadIter}
Y.~Saad.
\newblock {\em Iterative methods for sparse linear systems}.
\newblock Society for Industrial and Applied Mathematics, Philadelphia, PA,
  second edition, 2003.

\bibitem{2017arXiv170602205S}
F.~{Sch{\"a}fer}, T.~J. {Sullivan}, and H.~{Owhadi}.
\newblock {Compression, inversion, and approximate PCA of dense kernel matrices
  at near-linear computational complexity}.
\newblock {\em ArXiv e-prints}, June 2017.

\bibitem{MR1436437}
P.~Wojtaszczyk.
\newblock {\em A mathematical introduction to wavelets}, volume~37 of {\em
  London Mathematical Society Student Texts}.
\newblock Cambridge University Press, Cambridge, 1997.

\bibitem{Xu92}
J.~Xu.
\newblock Iterative methods by space decomposition and subspace correction.
\newblock {\em SIAM Rev.}, 34(4):581--613, 1992.

\bibitem{yserentant_1993}
H.~Yserentant.
\newblock Old and new convergence proofs for multigrid methods.
\newblock {\em Acta Numer.}, 2:285--326, 1993.

\bibitem{Yurinskii1986}
V.~V. Yurinski{\u\i}.
\newblock Averaging of symmetric diffusion in a random medium.
\newblock {\em Sibirsk. Mat. Zh.}, 27(4):167--180, 215, 1986.

\end{thebibliography}
\end{document}